\documentclass [a5paper,10pt,article,oneside]{memoir}
\usepackage{pages}
\acinq
\aquatre

\usepackage[english]{babel}
\usepackage[utf8]{inputenc}
\usepackage{amsmath}
\usepackage{amsfonts,mathrsfs,paralist,amssymb,bm,amsthm}
\usepackage{graphicx,color,paralist}

\usepackage{hyperref}
\hypersetup{colorlinks=false,hidelinks=true}
\usepackage{mathrsfs}
\usepackage[all]{xy}
\xyoption{poly}
\usepackage{biblatex}
\bibliography{LLN.bib}

\newcommand{\addpoint}[1]{#1\ ---\ }

\setsecnumdepth{subsubsection}

\setsecnumformat{\csname the#1\endcsname---}
\setsubsechook{\setsecnumformat{\csname the##1\endcsname\ ---\ }}
\setsechook{\setsecnumformat{\csname the##1\endcsname---}}

\setsecheadstyle{\centering\Large\bfseries\sffamily}
\setsubsecheadstyle{\large\bfseries\sffamily}
\setsubsubsecheadstyle{\itshape\addpoint}
\setsubsubsecindent{1em}
\setbeforesubsubsecskip{0em}
\setsubsubsechook{\setsecnumformat{{\normalfont\csname
      the##1\endcsname\ }}}
\setaftersubsubsecskip{-0em}
\setsubparaheadstyle{\itshape}
\newtheoremstyle{thm}
     {1.5ex plus .3ex minus .1ex}
     {1ex plus .3ex minus .1ex}
     {\itshape}
     {}
     {\bfseries\sffamily}
     {---}
     {0em}
     {$\bullet$\hbox{\ }#1\hbox{\ }#2}

\theoremstyle{thm}

\newtheorem{definition}{Definition}[section]
\newtheorem{theorem}[definition]{Theorem}
\newtheorem{lemma}[definition]{Lemma}
\newtheorem{proposition}[definition]{Proposition}
\newtheorem{corollary}[definition]{Corollary}

\newtheoremstyle{note}
     {1ex plus .3ex minus .1ex}
     {1ex plus .3ex minus .1ex}
     {}
     {}
     {\itshape}
     {.}
     {1em}
     {}
\theoremstyle{note}

\newlength{\remaining}
\newenvironment{intermediate}[1]
	{\par\medskip\noindent\makebox[2em][t]{\hfill#1\hfill}\begin{minipage}[t]{\the\remaining}\itshape}
          {\end{minipage}\par\medskip}

\newcommand{\tq}{\;:\;}
\newcommand{\iid}{{\normalfont\slshape\textsf{\small i.i.d.}}}

\def\calli#1{\expandafter\def\csname
  #1\endcsname{\mathcal{#1}}}	
\def\sets#1{\expandafter\def\csname
  bb#1\endcsname{\mathbb{#1}}}	
\def\rebar#1{\expandafter\def\csname #1bar\endcsname{\overline{\csname
      #1\endcsname}}}		
\def\gothify#1{\expandafter\def\csname
  #1#1#1\endcsname{\mathfrak{#1}}}	

\sets{R}	
\sets{N}	
\calli{I}	
\calli{M}	
\calli{T}	
\calli{V}
\calli{U}
\calli{W}
\gothify{F}	
\gothify{G}	
\rebar{M}	

\newcommand{\C}{\mathscr{C}}
\newcommand{\Cnv}{\mathfrak{C}}

\newcommand{\up}[1]{\,\uparrow #1}

\newcommand{\pr}{\mathbb{P}}
\newcommand{\esp}{\mathbb{E}}

\newcommand{\Cstar}{\mathfrak{C}}
\newcommand{\slgb}{\mbox{$\sigma$-al}\-ge\-bra}
\newcommand{\pas}{\text{$\pr$-a.s.}}

\newcommand{\m}[1]{|#1|}
\newcommand{\CF}{\textsf{CF}}
\newcommand{\BM}{\partial\M}
\newcommand{\lub}{\textsl{lub}}

\newcommand{\rest}[1]{\bigl|_{#1}}
\newcommand{\height}{\tau}

\renewcommand{\ast}{\textsf{\textsl{AST}}}
\newcommand{\un}{\mathbf{1}}
\newcommand{\scal}[2]{\langle#1,#2\rangle}

\tightlists

\newlength{\firstl}
\newlength{\secondl}
\newlength{\finall}

\newcommand{\firstc}{A Cut-Invariant Law of Large}
\newcommand{\secondc}{Numbers for Random Heaps}
\newcommand{\format}[1]{\makebox[\finall][s]{#1}} \newsavebox{\firstb}
\newsavebox{\secondb}

\newcommand{\lfaast}{%
  \savebox{\firstb}{\firstc}%
  \savebox{\secondb}{\secondc}%
  \setlength{\firstl}{\wd\firstb}%
  \setlength{\secondl}{\wd\secondb}%
  \ifnum\firstl>\secondl\setlength{\finall}{\firstl}\else
  \setlength{\finall}{\secondl}\fi\par
  \format{\firstc}\\
\format{\secondc}%
  }

\everymath{\normalfont}

\begin{document}
\mainmatter
\strut\vspace{-4em}
\begin{center}
\huge\bfseries\sffamily
\lfaast\par
\normalfont
\end{center}

\bigskip

\begin{center}
\begin{tabular}{c}
  \Large\sffamily Samy Abbes\\
  \small University Paris Diderot -- Paris~7\\
  \small CNRS Laboratory PPS (UMR 7126)\\
  \small Paris, France\\
  \small
  \ttfamily\footnotesize  samy.abbes@univ-paris-diderot.fr
\end{tabular}

\bigskip\bigskip
February 2015\par\bigskip\bigskip
\end{center}

\begin{abstract}
  Heap monoids equipped with Bernoulli measures are a model of
  probabilistic asynchronous systems.  We introduce in this framework
  the notion of asynchronous stopping time, which is analogous to the
  notion of stopping time for classical probabilistic processes. A
  Strong Bernoulli property is proved. A notion of cut-invariance is
  formulated for convergent ergodic means. Then a version of the
  Strong law of large numbers is proved for heap monoids with
  Bernoulli measures. Finally, we study a sub-additive version of the
  Law of large numbers in this framework based on Kingman sub-additive
  Ergodic Theorem.
\end{abstract}

\section{Introduction}
\label{sec:introduction}

Heaps of pieces are combinatorial structures that appear in several
domains of Combinatorics and Computer science. They were first studied
under the algebraic presentation of free partially commutative
monoids~\cite{cartier69}, also called trace monoids.  The visual
presentation as \emph{heaps of pieces} was introduced by
Viennot~\cite{viennot86}. Their application to computation models
relies on their ability to model in an intrinsic way the
\emph{asynchrony} of actions, that is to say, the fact that different
actions depending on disjoint resources may occur
concurrently~\cite{diekert90,diekert95}.

Several probabilistic models attached to heaps have been studied. A
study initiated by Vershik~\cite{VNBi,malyutin05} concerns the limit
behavior of random walks defined on heaps of pieces, when the size of
the heap monoid increases. In an other
approach~\cite{saheb89,krob03,bertoni08}, authors consider various
families of finite uniform probability distributions: on heaps of
size~$n$ and on heaps of height~$n$.

In this paper, we adopt a different point of view, and consider
Bernoulli measures on the \emph{boundary at infinity} of the heap
monoid. The elements of the boundary at infinity identify with
infinite heaps. The existence of Bernoulli measures for heap monoids
was proved in a joint work by the author and
J.~Mairesse~\cite{abbesmair14}. Bernoulli measures have the property
of being multiplicative with respect to the monoid structure of the
heap monoid, and they enjoy several properties related to the
combinatorial structure of heap monoids. In particular, elements of a
heap monoid are know to have a canonical normal form, called the
Cartier-Foata decomposition. Under a Bernoulli measure, the random
elements that successively occur in the decomposition form a Markov
chain, of which both the initial distribution and the transition
matrix have specific expressions, formulated through the combinatorial
tool of the M\"obius transform.

This paper introduces a formalism in order to express a Law of large
numbers for heap monoids under Bernoulli measures. A central
difficulty with heap monoids is that a given heap has several
presentations as a succession of pieces piled up one upon
another. Different presentations of the same heap differ in the order
of occurrences of certain pieces. The same issue occurs of course with
infinite heaps.  Consequently, there is no natural identification
between a given infinite heap, and an infinite sequence of
pieces. Nevertheless, if $\phi:\Sigma\to\bbR$ is a real valued
function defined on the set of basic pieces, we wish to obtain
asymptotic estimates for ergodic sums of the form
$S_n\phi=\phi(a_1)+\ldots+ \phi(a_n)$, where $(a_1,a_2,\ldots)$ is
\emph{one} presentation of a typical infinite heap~$\xi$.

For this purpose, we introduce a notion of \emph{cut} for infinite
heaps. Cuts are shown to share several properties with classical
stopping times, hence we actually call them \emph{asynchronous
  stopping times\/} (\ast). An \ast\ allows to select from any given
infinite heap~$\xi$, a sub-heap of~$\xi$. The analogy with standard
stopping times is this one: for instance the first hitting time $T$ of
a given state $a$ of a Markov chain takes an infinite path
$\omega=(x_1,x_2,\ldots)$ as input, and returns an
integer~$T(\omega)$. All the properties of the stopping time $T$ can
be interpreted when considering instead of the integer~$T(\omega)$,
the \emph{sub-path} $(x_1,\ldots,x_{T(\omega)})$. Hence a stopping
time can be seen as a particular kind of mapping from infinite paths
to paths. Similarly, an \ast\ is a particular kind of mapping from
infinite heaps to heaps.

Asynchronous stopping times can be iterated, just as usual stopping
times can be iterated. If $V$ is an \ast, we define the associated
sequence of iterated stopping times $(V_n)_{n\geq1}$ by recursively
piling up $V$-shaped heaps. The associated ergodic sums are then
defined by $S_{V,n}\phi=\phi V_n$\,, extending the action of $\phi$
from pieces to heaps by additivity. Let $\un$ be the constant function
equal to $1$ on every piece. By additivity, $\un V$ is the number of
pieces within heap~$V$. Then we show that, for every function~$\phi$,
the \emph{ergodic means}:
\begin{equation}
  \label{eq:4}
  \frac{S_{V,n}\phi}{S_{V,n}\un}
\end{equation}
have almost surely a limit when $n$ goes to infinity; and that this
limit \emph{does not depend on~$V$}. Furthermore, the limit can be
computed by using the stationary measure of the Markov chain given by
the elements of the Cartier-Foata decomposition.

Consider the case where the function $\phi$ is the characteristic
function of a certain piece, hence gives value $1$ to that piece and
$0$ to all other pieces.  Then the limit of the ergodic
means~(\ref{eq:4}) appears as the asymptotic density in large heaps of
the selected piece. The invariance with respect to the \ast\ $V$ means
that, whatever cut shape $V$ we choose, when one measures the density
of presence of a given piece within sub-heaps recursively obtained by
adding ``$V$-shaped'' heaps, the resulting asymptotic density is
always the same. This constitutes the cut-invariant Law of large
numbers.

We also give a sub-additive variant of the Law of large numbers for
random heaps. It is motivated by the existence of an interesting
sub-additive function on heaps, namely their \emph{height}.

We have included a preliminary section which illustrates most of the
notions on a very simple case, allowing to do all the calculations by
hand. In particular, the cut-invariance is demonstrated by performing
simple computations using only geometric laws. Later in the paper, all
the specific computations for this example will be re-interpreted
under the light of Bernoulli measures on heap monoids.

The paper is organized as follows. Section~\ref{sec:cut-invariance-an}
is the preliminary example section, which relies on no theoretical
material at all. Section~\ref{sec:heap-mono-bern} introduces the
background on heap monoids and Bernoulli
measures. Section~\ref{sec:asynchr-stopp-times} introduces
asynchronous stopping times for heap
monoids. Section~\ref{sec:iter-asynchr-stopp} introduces the iteration
of asynchronous stopping times.  Section~\ref{sec:cut-invariant-law}
states and proves the cut-invariant Law of large
numbers. Section~\ref{sec:cut-invariant-sub} is devoted to the
sub-additive variant of the Law of large numbers.

\section{Cut-Invariance on an Example}
\label{sec:cut-invariance-an}

The purpose of this preliminary section is twofold. First it will help
motivating the model of Bernoulli measures on trace monoids, which
will appear as a natural probabilistic model for systems involving
asynchronous actions. Second, it will illustrate that asymptotic
quantities relative to the model may be computed according to
different presentations, corresponding to different cut shapes of
random heaps.  The interesting point is that, whatever choice is made
for the shape of cuts, the associated asymptotic quantities are
invariant. The rest of the paper develops theoretical results
explaining this invariance, which we merely observe on a simple
example in this section.

\medskip Consider two communicating devices $A$ and~$B$. Device $A$
may perform some actions on its own. These actions will be called
\emph{type $a$ actions}. Similarly, device $B$ may perform actions of
type $b$ on its own. Finally, both devices may perform together a
synchronizing action of type~$c$, involving communication on both
sides---a sort of \emph{check-hand} action.

Consider the following simple probabilistic protocol, involving two
fixed probabilistic parameters $\lambda,\lambda'\in(0,1)$.
\begin{enumerate}
\item\label{item:1} Device $A$ and device $B$ perform actions of type
  $a$ and $b$ respectively, in an asynchronous and probabilistically
  independent way. The number $N_a$ of occurrences of type $a$ actions
  and the number $N_b$ of occurrences of type $b$ actions follow
  geometric laws with parameters $\lambda$ and~$\lambda'$
  respectively. Hence:
  \begin{equation}
    \label{eq:1}
    \forall k,k'\geq0\qquad\pr(N_a=k,\;N_b=k')=\lambda(1-\lambda)^k\cdot\lambda'(1-\lambda')^{k'}\,.
  \end{equation}
\item\label{item:2} Then devices $A$ and $B$ perform a synchronizing
  action of type~$c$, acknowledging that they have completed their
  local actions.
\item Go to~\ref{item:1}.
\end{enumerate}

We say that Steps~\ref{item:1}--\ref{item:2} form a \emph{round} of
the protocol.  The successive geometric variables that will occur when
executing several rounds of the protocol are assumed to be
independent.

The question that will guide us throughout this study is the
following: what are the asymptotic densities of actions of type $a$,
$b$ and~$c$? Hence, we are looking for non-negative quantities
$\gamma_a,\gamma_b,\gamma_c$ such that $\gamma_a+\gamma_b+\gamma_c=1$,
and that represent the average ratio of each type of action among all
three possible types. We shall see that there are various possible
definitions for the density vector
$\gamma=\begin{pmatrix}\gamma_a&\gamma_b&\gamma_c\end{pmatrix}$, and
we will observe by performing the different computations, that all
definitions lead to the same result. The core of the paper will
provide a general framework in which the notion of cut-invariance
gives a deep explanation for the equality of the results.  The above
protocol is the simplest possible example of this kind.

Before we suggest possible definitions for the asymptotic density
vector, it is customary to interpret the executions of the above
protocol with a \emph{heap of pieces} model. For this, associate
dominoes to each type of action~$a$, $b$ and~$c$. The occurrence of an
action corresponds to a domino of the associated type falling from top
to bottom until it either reaches the ground or a previously piled
domino. Asynchrony of types $a$ and $b$ actions is rendered by letting
dominoes of type $a$ and $b$ falling according to parallel, separated
lanes; whereas dominoes of type $c$ are blocking for dominoes of types
$a$ and~$b$, which renders the synchronization role of type $c$
actions. A typical first round of the protocol, in the heap model,
corresponds to a heap as depicted in Figure~\ref{fig:poqpjq}--$(a)$
for $N_a=1$ and $N_b=2$. The execution of several rounds of the
protocol makes the heap growing up, as depicted in
Figure~\ref{fig:poqpjq}--$(b)$.

Letting the protocol execute without limit of time yields random
\emph{infinite heaps}. Let $\pr$ denote the probability measure that
equips the canonical space associated to the execution of infinitely
many rounds of the protocol. The measure $\pr$ can also be seen as the
law of the infinite heap resulting from the execution of the protocol.

\begin{figure}
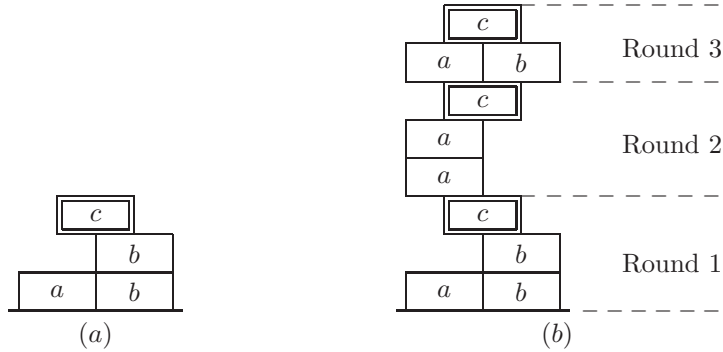

  \centering
\begin{gather*}
\begin{array}{ccc}
   \xy 
<.09em,0em>:
0="G",
"G"+(16,8)*{a},
"G";"G"+(32,0)**@{-};"G"+(32,16)**@{-};"G"+(0,16)**@{-};"G"**@{-},
(32,0)="G",
"G"+(16,8)*{b},
"G";"G"+(32,0)**@{-};"G"+(32,16)**@{-};"G"+(0,16)**@{-};"G"**@{-},
(32,16)="G",
"G"+(16,8)*{b},
"G";"G"+(32,0)**@{-};"G"+(32,16)**@{-};"G"+(0,16)**@{-};"G"**@{-},
(16,32)="G",
"G"+(16,8)*{c},
"G";"G"+(32,0)**@{-};"G"+(32,16)**@{-};"G"+(0,16)**@{-};"G"**@{-},
"G"+(2,2);"G"+(30,2)**@{-};"G"+(30,14)**@{-};"G"+(2,14)**@{-};"G"+(2,2)**@{-},
(-4,0);(68,0)**@{-},
\endxy
&\hspace{6em}\strut&
  \xy 
<.09em,0em>:
(-4,0);(68,0)**@{-},
0="G",
"G"+(16,8)*{a},
"G";"G"+(32,0)**@{-};"G"+(32,16)**@{-};"G"+(0,16)**@{-};"G"**@{-},
(32,0)="G",
"G"+(16,8)*{b},
"G";"G"+(32,0)**@{-};"G"+(32,16)**@{-};"G"+(0,16)**@{-};"G"**@{-},
(32,16)="G",
"G"+(16,8)*{b},
"G";"G"+(32,0)**@{-};"G"+(32,16)**@{-};"G"+(0,16)**@{-};"G"**@{-},
(16,32)="G",
"G"+(16,8)*{c},
"G";"G"+(32,0)**@{-};"G"+(32,16)**@{-};"G"+(0,16)**@{-};"G"**@{-},
"G"+(2,2);"G"+(30,2)**@{-};"G"+(30,14)**@{-};"G"+(2,14)**@{-};"G"+(2,2)**@{-},
(0,48)="G",
"G"+(16,8)*{a},
"G";"G"+(32,0)**@{-};"G"+(32,16)**@{-};"G"+(0,16)**@{-};"G"**@{-},
(0,64)="G",
"G"+(16,8)*{a},
"G";"G"+(32,0)**@{-};"G"+(32,16)**@{-};"G"+(0,16)**@{-};"G"**@{-},
(16,80)="G",
"G"+(16,8)*{c},
"G";"G"+(32,0)**@{-};"G"+(32,16)**@{-};"G"+(0,16)**@{-};"G"**@{-},
"G"+(2,2);"G"+(30,2)**@{-};"G"+(30,14)**@{-};"G"+(2,14)**@{-};"G"+(2,2)**@{-},
(0,96)="G",
"G"+(16,8)*{a},
"G";"G"+(32,0)**@{-};"G"+(32,16)**@{-};"G"+(0,16)**@{-};"G"**@{-},
(32,96)="G",
"G"+(16,8)*{b},
"G";"G"+(32,0)**@{-};"G"+(32,16)**@{-};"G"+(0,16)**@{-};"G"**@{-},
(16,112)="G",
"G"+(16,8)*{c},
"G";"G"+(32,0)**@{-};"G"+(32,16)**@{-};"G"+(0,16)**@{-};"G"**@{-},
"G"+(2,2);"G"+(30,2)**@{-};"G"+(30,14)**@{-};"G"+(2,14)**@{-};"G"+(2,2)**@{-},
(74,0);(130,0)**@{--},
(48,48);(130,48)**@{--},
(48,96);(130,96)**@{--},
(48,128);(130,128)**@{--},
(90,20)*{\rlap{Round $1$}},
(90,70)*{\rlap{Round $2$}},
(90,110)*{\rlap{Round $3$}},
\endxy
\\
(a)&&(b)
\end{array}
\end{gather*}
\caption{\textsl{Heaps of pieces corresponding to the execution of: $(a)$~the
  first round of the protocol, $(b)$~the three first rounds of the protocol.}}
  \label{fig:poqpjq}
\end{figure}

It is important to observe that the law $\pr$ cannot be reached by the
execution of any Markov chain with three states $a,b,c$ (proof left to
the reader for this particular example; or to be deduced from the
results of~\S~\ref{sec:bernoulli-measures}). In particular, the
estimation of the asymptotic quantities that we perform below do not
result from a straightforward translation into a Markov chain model.

If $N$ denotes the total number of pieces at Round~1 of the protocol,
one has:
\begin{align*}
  N&=N_a+N_b+N_c\,,&\text{with }N_c&=1\,.
\end{align*}

Each round of the protocol corresponding to a fresh pair $(N_a,N_b)$,
it is natural to define the asymptotic density vector~$\gamma$ as:
\begin{align}
\label{eq:2}
  \gamma_a&=\frac{\esp N_a}{\esp N}\,,&
  \gamma_b&=\frac{\esp N_b}{\esp N}\,,&
  \gamma_c&=\frac{\esp N_c}{\esp N}=\frac1{\esp N}\,,
\end{align}
where $\esp$ denotes the expectation with respect to
probability~$\pr$.

Since $N_a$ and $N_b$ follow geometric laws on the one hand, and since
$\esp N=1+\esp N_a+\esp N_b$ on the other hand, the computation of
$\gamma=\begin{pmatrix}\gamma_a&\gamma_b&\gamma_c\end{pmatrix}$ is
immediate and yields:
\begin{align}
  \label{eq:3}
\gamma_a&=\frac{\lambda'(1-\lambda)}{\lambda+\lambda'-\lambda\lambda'}\,,&
\gamma_b&=\frac{\lambda(1-\lambda')}{\lambda+\lambda'-\lambda\lambda'}\,,&
\gamma_c&=\frac{\lambda\lambda'}{\lambda+\lambda'-\lambda\lambda'}\,.
\end{align}

The description we have given of the protocol has naturally lead us to
the definition~(\ref{eq:2}) for the density vector~$\gamma$. However,
abstracting from the description of the protocol and focusing on the
heap model only, we realize that dominoes $c$ play a particular role,
which corresponds to an asymmetry between the three types of
actions. The special role of $c$ lies in the following: at each round,
the formed pile ends up with a type $c$ domino. Hence, the specific
formulation of the protocol that we have adopted above can be
rephrased as follows:
\begin{enumerate}
\item\label{item:3} Consider the probability distribution $\pr$ over random heaps.
\item\label{item:4} Recursively cut an infinite random heap, say $\xi$ distributed
  according to~$\pr$, by selecting the successive occurrences of type
  $c$ dominoes in~$\xi$, and cutting apart the associated sub-heaps, as in
  Figure~\ref{fig:poqpjq}--$(b)$.
\end{enumerate}

In this new formulation, point~\ref{item:3} is now intrinsic; while
only point~\ref{item:4} relies on a special cut shape. Henceforth, the
following questions are natural: if we change the cut shape in
point~\ref{item:4}, and if we compute the new associated densities of
pieces, say
$\gamma'=\begin{pmatrix}\gamma'_a&\gamma'_b&\gamma'_c\end{pmatrix}$,
is it true that $\gamma=\gamma'$?  And is there a more symmetric model
to describe random heaps, that would not give a special role to any
type of domino, and that could also provide a way for computing
densities?

Before we enter the core of the topic in the next sections of the
paper, we shall merely conclude this introductory section by defining
and computing indeed an alternative density vector~$\gamma'$, and
obtain by simple computation the equality $\gamma=\gamma'$\,.

\begin{figure}
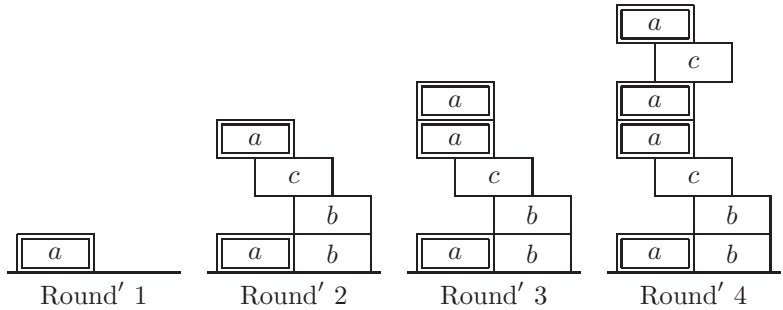

  \centering
  \begin{gather*}
\begin{array}{cccc}
   \xy 
<.09em,0em>:
(-4,0);(68,0)**@{-},
0="G",
"G"+(16,8)*{a},
"G";"G"+(32,0)**@{-};"G"+(32,16)**@{-};"G"+(0,16)**@{-};"G"**@{-},
"G"+(2,2);"G"+(30,2)**@{-};"G"+(30,14)**@{-};"G"+(2,14)**@{-};"G"+(2,2)**@{-},
\endxy
&
  \xy 
<.09em,0em>:
(-4,0);(68,0)**@{-},
0="G",
"G"+(16,8)*{a},
"G";"G"+(32,0)**@{-};"G"+(32,16)**@{-};"G"+(0,16)**@{-};"G"**@{-},
"G"+(2,2);"G"+(30,2)**@{-};"G"+(30,14)**@{-};"G"+(2,14)**@{-};"G"+(2,2)**@{-},
(32,0)="G",
"G"+(16,8)*{b},
"G";"G"+(32,0)**@{-};"G"+(32,16)**@{-};"G"+(0,16)**@{-};"G"**@{-},
(32,16)="G",
"G"+(16,8)*{b},
"G";"G"+(32,0)**@{-};"G"+(32,16)**@{-};"G"+(0,16)**@{-};"G"**@{-},
(16,32)="G",
"G"+(16,8)*{c},
"G";"G"+(32,0)**@{-};"G"+(32,16)**@{-};"G"+(0,16)**@{-};"G"**@{-},
(0,48)="G",
"G"+(16,8)*{a},
"G";"G"+(32,0)**@{-};"G"+(32,16)**@{-};"G"+(0,16)**@{-};"G"**@{-},
"G"+(2,2);"G"+(30,2)**@{-};"G"+(30,14)**@{-};"G"+(2,14)**@{-};"G"+(2,2)**@{-},
\endxy
&
  \xy 
<.09em,0em>:
(-4,0);(68,0)**@{-},
0="G",
"G"+(16,8)*{a},
"G";"G"+(32,0)**@{-};"G"+(32,16)**@{-};"G"+(0,16)**@{-};"G"**@{-},
"G"+(2,2);"G"+(30,2)**@{-};"G"+(30,14)**@{-};"G"+(2,14)**@{-};"G"+(2,2)**@{-},
(32,0)="G",
"G"+(16,8)*{b},
"G";"G"+(32,0)**@{-};"G"+(32,16)**@{-};"G"+(0,16)**@{-};"G"**@{-},
(32,16)="G",
"G"+(16,8)*{b},
"G";"G"+(32,0)**@{-};"G"+(32,16)**@{-};"G"+(0,16)**@{-};"G"**@{-},
(16,32)="G",
"G"+(16,8)*{c},
"G";"G"+(32,0)**@{-};"G"+(32,16)**@{-};"G"+(0,16)**@{-};"G"**@{-},
(0,48)="G",
"G"+(16,8)*{a},
"G";"G"+(32,0)**@{-};"G"+(32,16)**@{-};"G"+(0,16)**@{-};"G"**@{-},
"G"+(2,2);"G"+(30,2)**@{-};"G"+(30,14)**@{-};"G"+(2,14)**@{-};"G"+(2,2)**@{-},
(0,64)="G",
"G"+(16,8)*{a},
"G";"G"+(32,0)**@{-};"G"+(32,16)**@{-};"G"+(0,16)**@{-};"G"**@{-},
"G"+(2,2);"G"+(30,2)**@{-};"G"+(30,14)**@{-};"G"+(2,14)**@{-};"G"+(2,2)**@{-},
\endxy
&
  \xy 
<.09em,0em>:
(-4,0);(68,0)**@{-},
0="G",
"G"+(16,8)*{a},
"G";"G"+(32,0)**@{-};"G"+(32,16)**@{-};"G"+(0,16)**@{-};"G"**@{-},
"G"+(2,2);"G"+(30,2)**@{-};"G"+(30,14)**@{-};"G"+(2,14)**@{-};"G"+(2,2)**@{-},
(32,0)="G",
"G"+(16,8)*{b},
"G";"G"+(32,0)**@{-};"G"+(32,16)**@{-};"G"+(0,16)**@{-};"G"**@{-},
(32,16)="G",
"G"+(16,8)*{b},
"G";"G"+(32,0)**@{-};"G"+(32,16)**@{-};"G"+(0,16)**@{-};"G"**@{-},
(16,32)="G",
"G"+(16,8)*{c},
"G";"G"+(32,0)**@{-};"G"+(32,16)**@{-};"G"+(0,16)**@{-};"G"**@{-},
(0,48)="G",
"G"+(16,8)*{a},
"G";"G"+(32,0)**@{-};"G"+(32,16)**@{-};"G"+(0,16)**@{-};"G"**@{-},
"G"+(2,2);"G"+(30,2)**@{-};"G"+(30,14)**@{-};"G"+(2,14)**@{-};"G"+(2,2)**@{-},
(0,64)="G",
"G"+(16,8)*{a},
"G";"G"+(32,0)**@{-};"G"+(32,16)**@{-};"G"+(0,16)**@{-};"G"**@{-},
"G"+(2,2);"G"+(30,2)**@{-};"G"+(30,14)**@{-};"G"+(2,14)**@{-};"G"+(2,2)**@{-},
(16,80)="G",
"G"+(16,8)*{c},
"G";"G"+(32,0)**@{-};"G"+(32,16)**@{-};"G"+(0,16)**@{-};"G"**@{-},
(0,96)="G",
"G"+(16,8)*{a},
"G";"G"+(32,0)**@{-};"G"+(32,16)**@{-};"G"+(0,16)**@{-};"G"**@{-},
"G"+(2,2);"G"+(30,2)**@{-};"G"+(30,14)**@{-};"G"+(2,14)**@{-};"G"+(2,2)**@{-},
\endxy
\\
\text{Round}'\ 1&\text{Round}'\ 2&\text{Round}'\ 3&\text{Round}'\ 4
\end{array}
\end{gather*}
  \caption{\textsl{Cutting heaps relatively to type $a$ dominoes.}}
  \label{fig:paopajka}
\end{figure}

We consider the variant where heaps are cut at ``first occurrence'' of
type $a$ dominoes. We illustrate in Figure~\ref{fig:paopajka} the
successive new rounds, corresponding to the same heap that we already
depicted in Figure~\ref{fig:poqpjq}. Observe that each new round
involves a finite but unbounded number of rounds of the original
protocol.

Let $V'$ denote the random heap obtained by cutting an infinite heap
at the first occurrence of a domino of type~$a$ , which is defined
with $\pr$-probability~$1$. Denote by $|V'|$ the number of pieces
in~$V'$. Denote also by $|V'|_a$ the number of occurrences of
piece $a$ in~$V'$\,, and so on for $|V'|_b$ and for~$|V'|_c$\,, so
that $|V'|=|V'|_a+|V'|_b+|V'|_c$ holds. By construction, $|V'|_a=1$
holds $\pr$-almost surely.  We define the new density vector
$\gamma'=\begin{pmatrix}\gamma'_a&\gamma'_b&\gamma'_c\end{pmatrix}$
by:
\begin{align*}
  \gamma'_a&=\frac{\esp |V'|_a}{\esp |V'|}=\frac1{\esp |V'|}\,,& \gamma'_b&=\frac{\esp
    |V'|_b}{\esp |V'|}\,,& \gamma'_c&=\frac{\esp |V'|_c}{\esp |V'|}\,.
\end{align*}

The computation of $\gamma'$ is easy. A typical random heap $V'$ ends
up with~$a$, after having crossed, say, $k$~occurrences
of~$c$. Immediately before the~$j^{\text{th}}$ occurrence of~$c$, for
$j\leq k$, there has been an arbitrary number, say~$l_j$\,, of
occurrences of~$b$.  Hence: $V'=b^{l_1}\cdot c\cdot\ldots\cdot
b^{l_k}\cdot c\cdot a$, with $k\geq0$ and $l_1,\ldots,l_k\geq0$.
Referring to the definition of the probability~$\pr$, one has:
\begin{align*}
  \forall k, l_1,\ldots,l_k\geq0,\qquad
\pr(V'=b^{l_1}\cdot c\cdot\ldots\cdot b^{l_k}\cdot c\cdot a)=\lambda^k(1-\lambda) \lambda'^k(1-\lambda')^{l_1+\ldots+l_k}\,.
\end{align*}

Since $|V'|_b=l_1+\ldots+l_k$ and $|V'|_c=k$, the computation of the
various expectations is straightforward:
\begin{align*}
  \esp
  |V'|_b&=\sum_{k,l_1,\ldots,l_k\geq0}(l_1+\ldots+l_k)\lambda^k(1-\lambda)
  \lambda'^k(1-\lambda')^{l_1+\ldots+l_k}\\
  &=(1-\lambda)\sum_{k\geq0}(\lambda\lambda')^k
  k\sum_{l_1\geq0}l_1(1-\lambda')^{l_1}\Bigl(\sum_{l\geq0}(1-\lambda')^l\Bigr)^{k-1}\\
&=\frac{\lambda(1-\lambda')}{\lambda'(1-\lambda)}\,.\displaybreak[0]
\\
\esp |V'|_c&=\sum_{k,l_1,\ldots,l_k\geq0}k\lambda^k(1-\lambda)
  \lambda'^k(1-\lambda')^{l_1+\ldots+l_k}\\
&=(1-\lambda)\sum_{k\geq0}k(\lambda\lambda')^k\Bigl(\sum_{l\geq0}(1-\lambda')^l\Bigr)^k\\
&=\frac\lambda{1-\lambda}\,.\displaybreak[0]
\\
\esp |V'|&=1+\esp |V'|_b+\esp |V'|_c=\frac{\lambda+\lambda'-\lambda\lambda'}{\lambda'(1-\lambda)}
\end{align*}

We obtain the density vector $\gamma'$:
\begin{align*}
  \gamma'_a&=\frac{\lambda'(1-\lambda)}{\lambda+\lambda'-\lambda\lambda'}\,,&
\gamma'_b&=\frac{\lambda(1-\lambda')}{\lambda+\lambda'-\lambda\lambda'}\,,&
\gamma'_c&=\frac{\lambda\lambda'}{\lambda+\lambda'-\lambda\lambda'}\,.
\end{align*}
Comparing with~(\ref{eq:3}), we observe the announced equality
$\gamma=\gamma'$\,.

\section{Heap Monoids and Bernoulli Measures}
\label{sec:heap-mono-bern}

In this section, we collect the needed material on heap monoids and on
associated Bernoulli measures. Classical references on heap monoids
are \cite{cartier69,viennot86,diekert90,diekert95}. For Bernoulli
measures, we refer to the original paper~\cite{abbesmair14}.

\subsection{Independence Pairs. Heap Monoids}
\label{sec:heap-mono-length}

Let $\Sigma$ be a finite, non empty set of cardinality $>1$. Elements
of $\Sigma$ are called \emph{pieces}. We say that the pair
$(\Sigma,I)$ is an \emph{independence pair} if $I$~is a symmetric and
irreflexive relation on~$\Sigma$, called \emph{independence relation}.
We will furthermore always assume that the following irreducibility
assumption is in force: the associated \emph{dependence relation} $D$
on~$\Sigma$, defined by $D=(\Sigma\times\Sigma)\setminus I$, makes the
graph $(\Sigma,D)$ connected.

The free monoid generated by $\Sigma$ is denoted by~$\Sigma^*$\,, it
consists of all $\Sigma$-words. The congruence $\I$ is defined as the
smallest congruence on $\Sigma^*$ that contains all pairs of the form
$(ab,ba)$ for $(a,b)$ ranging over~$I$. The \emph{heap monoid}
$\M=\M(\Sigma,I)$ is defined as the quotient monoid
$\M=\Sigma^*/\I$. Hence $\M$ is the presented monoid:
\begin{equation*}
  \M=\langle\Sigma\;|\; a b= b a\,, \text{ for } (a,b)\in I\rangle\,.
\end{equation*}

Elements of a heap monoid are called \emph{heaps}. In the literature,
heaps are also called \emph{traces}; heap monoids are also called
\emph{free partially commutative monoids}.

We denote by the dot ``$\cdot$'' the concatenation of heaps, and by
$0$ the empty heap. 

A graphical interpretation of heaps is obtained by letting pieces
fall as dominoes on a ground, in such a way that
\begin{inparaenum}[(1)]
\item dominoes corresponding to different occurrences of the same
  piece follow the same lane; and
\item two dominoes corresponding to pieces $a$ and $b$ are blocking
  with respect to each other if and only if $(a,b)\notin I$\,.
\end{inparaenum}
This is illustrated in Figure~\ref{fig:paopjalaa} for the heap monoid
on three generators $\T=\langle a,b,c\;|\; ab=ba\rangle$. This will be
our running example throughout the paper.

\begin{figure}
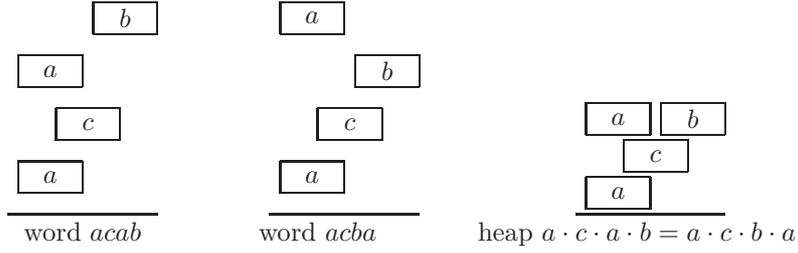

  \centering
  \begin{tabular}{ccc}
\xy
<.1em,0em>:
(0,6)="G",
"G"+(12,6)*{a},
"G";"G"+(24,0)**@{-};"G"+(24,12)**@{-};"G"+(0,12)**@{-};"G"**@{-},
(14,26)="G",
"G"+(12,6)*{c},
"G";"G"+(24,0)**@{-};"G"+(24,12)**@{-};"G"+(0,12)**@{-};"G"**@{-},
(0,46)="G",
"G"+(12,6)*{a},
"G";"G"+(24,0)**@{-};"G"+(24,12)**@{-};"G"+(0,12)**@{-};"G"**@{-},
(28,66)="G",
"G"+(12,6)*{b},
"G";"G"+(24,0)**@{-};"G"+(24,12)**@{-};"G"+(0,12)**@{-};"G"**@{-},
(-4,-2);(52,-2)**@{-}
\endxy
& \quad\qquad
\xy
<.1em,0em>:
(0,6)="G",
"G"+(12,6)*{a},
"G";"G"+(24,0)**@{-};"G"+(24,12)**@{-};"G"+(0,12)**@{-};"G"**@{-},
(14,26)="G",
"G"+(12,6)*{c},
"G";"G"+(24,0)**@{-};"G"+(24,12)**@{-};"G"+(0,12)**@{-};"G"**@{-},
(0,66)="G",
"G"+(12,6)*{a},
"G";"G"+(24,0)**@{-};"G"+(24,12)**@{-};"G"+(0,12)**@{-};"G"**@{-},
(28,46)="G",
"G"+(12,6)*{b},
"G";"G"+(24,0)**@{-};"G"+(24,12)**@{-};"G"+(0,12)**@{-};"G"**@{-},
(-4,-2);(52,-2)**@{-}
\endxy
&\qquad \xy
<.1em,0em>:
0="G",
"G"+(12,6)*{a},
"G";"G"+(24,0)**@{-};"G"+(24,12)**@{-};"G"+(0,12)**@{-};"G"**@{-},
(14,14)="G",
"G"+(12,6)*{c},
"G";"G"+(24,0)**@{-};"G"+(24,12)**@{-};"G"+(0,12)**@{-};"G"**@{-},
(0,28)="G",
"G"+(12,6)*{a},
"G";"G"+(24,0)**@{-};"G"+(24,12)**@{-};"G"+(0,12)**@{-};"G"**@{-},
(28,28)="G",
"G"+(12,6)*{b},
"G";"G"+(24,0)**@{-};"G"+(24,12)**@{-};"G"+(0,12)**@{-};"G"**@{-},
(-4,-2);(52,-2)**@{-}
\endxy\\
word $acab$&\quad word $acba$&\quad heap $a\cdot c\cdot a\cdot b=a\cdot c\cdot b\cdot a$
\end{tabular}
\caption{Two congruent words and the resulting heap}
  \label{fig:paopjalaa}
\end{figure}

\subsection{Length and Ordering}
\label{sec:mass-order-heaps}

The congruence $\I$ coincides with the reflexive and transitive
closure of the \emph{immediate equivalence}, which relates any two
$\Sigma$-words of the form $xaby$ and $xbay$, where $x,y\in\Sigma^*$
and $(a,b)\in I$. In particular, the length of congruent words is
invariant, which defines a mapping $|\,\cdot\,|:\M\to\bbN$\,. For any
heap $x\in\M$, the integer $|x|$ is called the
\emph{length} of~$x$. Obviously, the length is additive on~$\M$: $|x\cdot y|=|x|+|y|$
for all heaps $x,y\in\M$. 

The left divisibility relation ``$\leq$'' is defined on $\M$ by:
\begin{equation}
  \label{eq:5}
  \forall x,y\in\M\quad x\leq y\iff\exists z\in\M\quad y=x\cdot z\,.
\end{equation}
It defines a partial ordering relation on~$\M$. If $x\leq y$, we say
that $x$ is a \emph{sub-heap} of~$y$.  Note that holds: $x\leq
y\implies \m x\leq\m y$.

Visually, $x\leq y$ means that $x$ is a heap that can be seen at the
bottom of~$y$. But, contrary to words, a given heap might for instance
have several sub-heaps of length~$1$. Indeed, in the example monoid~$\T$
defined above, one has both $a\leq a\cdot b$ and $b\leq a\cdot b$ since
$a\cdot b=b\cdot a$ in~$\T$.

Heap monoids are known to be \emph{cancellative}, meaning:
\begin{equation*}
  \label{eq:6}
  \forall x,y,u,u'\in\M\quad
x\cdot u\cdot y=x\cdot u'\cdot y\implies u=u'\,.
\end{equation*}
This implies in particular that, if $x,y$ are heaps such that $x\leq
y$ holds, then the heap $z$ in~(\ref{eq:5}) is unique. We denote it
by: $z=y-x$\,.

\subsection{Cliques. Cartier-Foata Normal Form}
\label{sec:cliq-cart-foata}

Recall that a \emph{clique} of a graph is a sub-graph which is
complete as a graph---this includes the empty graph. The independence
pair $(\Sigma,I)$ may be seen as a graph. The cliques of $(\Sigma,I)$
are called the \emph{independence cliques}, or simply the
\emph{cliques} of the heap monoid~$\M$.

Each clique $\gamma$, with set of vertices $\{a_1,\ldots,a_n\}$\,,
identifies with the heap $a_1\cdot\ldots\cdot a_n\in\M$\,, which, by
commutativity, is independent of the sequence $(a_1,\ldots,a_n)$
enumerating the vertices of~$\gamma$. In the graphical representation
of heaps, cliques correspond to horizontal layers of pieces.  Note
that any piece is by itself a clique of length~$1$. We denote by $\C$
the set of cliques of the heap monoid~$\M$, and by
$\Cstar=\C\setminus\{0\}$ the set of non empty cliques.  For the
running example monoid~$\T$, there are 4 non empty cliques:
$\Cstar=\{a,b,c,a\cdot b\}$.

It is visually intuitive that heaps can be uniquely written as a
succession of horizontal layers, hence of cliques. More precisely,
define the relation $\to$ on $\C$ as follows:
\begin{equation*}
  \label{eq:7}
  \forall \gamma,\gamma'\in\C\quad \gamma\to \gamma'\iff\forall b\in \gamma'\quad\exists a\in
  \gamma\quad (a,b)\notin I\,. 
\end{equation*}

The relation $\gamma\to \gamma'$ means that $\gamma$ ``supports''~$\gamma'$, in the sense
that no piece of $\gamma'$ can fall when piled upon~$\gamma$.

A sequence $\gamma_1,\ldots,\gamma_n$ of cliques is said to be
\emph{Cartier-Foata admissible} if $\gamma_i\to \gamma_{i+1}$ holds
for all $i\in\{1,\ldots,n-1\}$\,. For every \emph{non empty heap}
$x\in\M$, there exists a unique integer $n\geq1$ and a unique
Cartier-Foata admissible sequence $(\gamma_1,\ldots,\gamma_n)$ of
\emph{non empty cliques} such that $x=\gamma_1\cdot\ldots\cdot
\gamma_n$\,.

This unique sequence of cliques is called the \emph{Cartier-Foata
  normal form} or \emph{decomposition} of~$x$ (\CF\ for short). The
integer $n$ is called the \emph{height} of~$x$, denoted by
$n=\height(x)$. By convention, we set $\height(0)=0$.

Heaps are thus in one-to-one correspondence with finite paths in the
graph $(\Cstar,\to)$ of \emph{non-empty} cliques. By convention, let
us extend any such finite path by infinitely many occurrences of the
\emph{empty} clique~$0$. Observe that $0$ is an absorbing vertex of
the graph of cliques $(\C,\to)$, since $0\to \gamma\iff\gamma=0$, and
$\gamma\to0$ holds for every $\gamma\in\C$. With this convention,
heaps are now in one-to-one correspondence with infinite paths in
$(\C,\to)$, that reach the $0$ node---and then stay in it.

\subsection{Infinite Heaps. Boundary}
\label{sec:infin-heaps.-bound}

We define an \emph{infinite heap} as any infinite admissible sequence
of cliques in the graph $(\C,\to)$, that does not reach the empty
clique.  The set of infinite heaps is called the \emph{boundary at
  infinity of~$\M$}, or simply the \emph{boundary} of~$\M$, and we
denote it by~$\BM$~\cite{abbesmair14}. By contrast, elements of $\M$
might be called \emph{finite heaps}. Extending the previous
terminology, we still refer to the cliques $\gamma_n$ such that
$\xi=(\gamma_n)_{n\geq1}$ as to the \CF\ decomposition of an infinite
heap~$\xi$\,.

It is customary to introduce the following notation:
\begin{equation*}
  \Mbar=\M\cup\BM\,.
\end{equation*}

Elements of $\Mbar$ are thus in one-to-one correspondence with
infinite paths in \mbox{$(\C,\to)$}; those that reach~$0$, correspond
to heaps, and those that do not reach~$0$, correspond to infinite
heaps. For $\xi\in\BM$\,, we put $|\xi|=\infty$\,. 

We wish to extend to $\Mbar$ the order $\leq$ previously defined
on~$\M$. For this, we use the representation of heaps, either finite
or infinite, as infinite paths in the graph $(\C,\to)$, and we put,
for $\xi=(\gamma_1,\gamma_2,\ldots)$ and
$\xi'=(\gamma'_1,\gamma'_2,\ldots)$:
\begin{gather}
\label{eq:22}
  \xi\leq\xi'\iff\forall n\geq1\quad \gamma_1\cdot\ldots\cdot\gamma_n\leq\gamma'_1\cdot\ldots\cdot\gamma'_n\,.
\end{gather}

\begin{proposition}[\mbox{\cite{abbesmair14}}]
\label{prop:1}
The relation defined in~\eqref{eq:22} makes $(\Mbar,\leq)$ a partial
order, which extends $(\M,\leq)$\,, and with the following properties:
\begin{enumerate}
\item\label{item:10} $(\Mbar,\leq)$ is complete with respect to:
  \begin{enumerate}
  \item\label{item:13} Least upper bounds (\lub) of non-decreasing
    sequences: for every sequence $(x_n)_{n\geq1}$ such that
    $x_n\in\Mbar$ and $x_n\leq x_{n+1}$ for all integers $n\geq1$, the
    \lub\ $\bigvee_{n\geq1} x_n$ exists in~$\Mbar$.
  \item\label{item:14} Greatest lower bounds of arbitrary subsets.
  \end{enumerate}
\item\label{item:11} For every heap $\xi\in\Mbar$, either finite or
  infinite, the following subset:
\begin{equation}
\label{eq:25}
  L(\xi)=\{x\in\Mbar\tq x\leq\xi\}\,,
\end{equation}
is a complete lattice with\/ $0$ and\/ $\xi$ as minimal and maximal
elements.
\item\label{item:12} {\normalfont(Finiteness property of elements of
    $\M$ in the sense of \cite{gierz03})}\quad For every finite heap
  $x\in\M$, and for every non-decreasing sequence $(x_n)_{n\geq0}$ of
  heaps, holds:
  \begin{gather*}
\bigvee_{n\geq0} x_n\geq x\implies\exists n\geq0\quad x_n\geq x\,.
  \end{gather*}
\item\label{item:17} The elements of $\M$ form a basis of $\Mbar$ in
  the sense of\/~{\normalfont\cite{gierz03}}: for all
  $\xi,\xi'\in\Mbar$ holds:
  \begin{gather*}
    \xi\geq\xi'\iff(\forall x\in\M\quad \xi'\geq x\implies \xi\geq
    x)\,.
  \end{gather*}

\end{enumerate}
\end{proposition}

\subsection{Elementary Cylinders. Bernoulli Measures}
\label{sec:elem-clyind-bern}

For $x\in\M$ a heap, the \emph{elementary cylinder of base~$x$} is the
following non empty subset of~$\BM$:
\begin{equation*}
  \label{eq:12}
  \up x=\{\xi\in\BM\tq x\leq \xi\}\,.
\end{equation*}

We equip the boundary $\BM$ with the \slgb:
\begin{gather*}
  \FFF=\sigma\langle\up x\,,\; x\in\M\rangle\,,
\end{gather*}
generated by the countable collection of
elementary cylinders. From now on, when referring to the
\emph{boundary}, we shall always mean the measurable space
$(\BM,\FFF)$.

We say that a probability measure $\pr$ on the boundary is a
\emph{Bernoulli measure} whenever it satisfies:
\begin{equation}
  \label{eq:13}
  \forall x,y\in\M\quad\pr\bigl(\up(x\cdot y)\bigr)=\pr(\up
  x)\cdot\pr(\up y)\,. 
\end{equation}

We shall furthermore impose the following condition to avoid
degenerated cases:
\begin{equation}
  \label{eq:14}
  \forall x\in\M\quad\pr(\up x)>0.
\end{equation}

If $\pr$ is a Bernoulli measure on the boundary, the positive function
\mbox{$f:\M\to\bbR$} defined by:
\begin{equation*}
  \forall x\in\M\quad f(x)=\pr(\up x)\,,
\end{equation*}
is called the \emph{valuation associated to\/~$\pr$.} By definition of
Bernoulli measures, $f$~is multiplicative: $f(x\cdot y)=f(x)\cdot
f(y)$. In particular, the values of $f$ on $\M$ are entirely
determined by the finite collection $(p_a)_{a\in\Sigma}$ of
\emph{characteristic numbers of\/~$\pr$} defined by the value of $f$
on single pieces:
\begin{equation}
\label{eq:15}
  \forall a\in\Sigma\quad p_a=f(a)\,.
\end{equation}

The condition~(\ref{eq:14}) is equivalent to impose $p_a>0$ for all
$a\in\Sigma$. 

For any heap $x\in\M$, $\pr(\up x)$~corresponds to the probability of
seeing $x$ at bottom of a random infinite heap with law~$\pr$\,. By
definition of Bernoulli measures, this probability is equal to the
product $p_{a_1}\times\dots\times p_{a_n}$\,, where the word $a_1\dots
a_n$ is any representative word of the heap~$x$.

\subsection{Interpretation of the Introductory Probabilistic Protocol}
\label{sec:interpr-intr-prob}

The sole definition of Bernoulli measures already allows us to
interpret the probabilistic protocol introduced
in~\S~\ref{sec:cut-invariance-an} by means of a Bernoulli measure
$\pr$ on the boundary of a heap monoid. Obviously, the heap monoid to
consider coincides with our running example $\T=\langle
a,b,c\;|\;ab=ba\rangle$ on three generators. Let us check that the
measure $\pr$ defined by the law of infinite heaps generated by the
described protocol is indeed Bernoulli.

Any  heap $x\in\T$ can be uniquely described under  the form: 
\begin{equation*}
x=(a^{r_1}\cdot b^{s_1})\cdot
c\cdot\ldots\cdot (a^{r_k}\cdot b^{s_k})\cdot c\cdot a^{r_{k+1}}\cdot
b^{s_{k+1}}\,,
\end{equation*}
for some integers $k,r_1,s_1,\ldots,r_{k+1},s_{k+1}\geq0$\,. For such
a heap~$x$, referring to the description of the probabilistic
protocol, the associated cylinder $\up x$ is described by:
\begin{equation*}
  \up x=
  \begin{cases}
\text{Round $1$:}&N_a=r_1,\ N_b=s_1\\
\setbox0=\hbox{Round $1$}\makebox[\wd0][c]{$\vdots$}\\
\text{Round $k$:}&N_a=r_k,\ N_b=s_k\\
\text{Round $k+1$:}&N_a\geq r_{k+1},\ N_b\geq s_{k+1}    
  \end{cases}
\end{equation*}
and is given probability:
\begin{align*}
  \pr(\up x)&=(\lambda\lambda')^k(1-\lambda)^{r_1+\ldots+r_k}
  (1-\lambda')^{s_1+\ldots+s_k}
  (1-\lambda)^{r_{k+1}}(1-\lambda')^{s_{k+1}}\,.
\end{align*}

If $f:\M\to\bbR$ is the multiplicative function defined by:
\begin{align}
\label{eq:27}
f(a)&=1-\lambda\,,& f(b)&=1-\lambda'\,,&f(c)&=\lambda\lambda'\,,
\end{align}
it is thus apparent that $\pr(\up x)=f(x)$ holds for all
$x\in\M$. Since $\pr(\up x)$ is multiplicative, the measure $\pr$ is
Bernoulli.  

In passing, we notice that, whatever the choices of
$\lambda,\lambda'\in(0,1)$, the equation:
\begin{gather*}
1-f(a)-f(b)-f(c)+f(a)f(b)=0
\end{gather*}
is satisfied, as one shall expect from~(\ref{eq:20})--$(a)$ below.

\subsection{Compatible Heaps}
\label{sec:compatible-heaps}

In the sequel, we shall often use the following facts. We say that two
heaps $x,y\in\M$ are \emph{compatible} if there exists $z\in\M$ such
that $x\leq z$ and $y\leq z$. It follows in particular from
Proposition~\ref{prop:1} that the following propositions are
equivalent:
\begin{enumerate}[(i)]
\item $x,y\in\M$ are compatible;
\item $\up x\;\cap\up y\neq\emptyset$;
\item the \lub\ $x\vee y$ exists in~$\M$.
\end{enumerate}
In this case, we also have:
\begin{align*}
\up x\;\cap\up y&=\up(x\vee y)\,,
\end{align*}
and if $\pr$ is a Bernoulli probability measure on~$\BM$, still for
$x$ and $y$ compatible:
\begin{align}
\label{eq:9}
  \pr\bigl(\up x\,|\up y\bigr)&=\pr\bigl(\up ((x\vee y)- y)\bigr)=\pr\bigl(\up(x-(x\wedge y))\bigr)\,.
\end{align}

\subsection{M\"obius Transform. Markov Chain of Cliques}
\label{sec:bernoulli-measures}

Call \emph{valuation} any positive and multiplicative function
$f:\M\to\bbR$; the valuations induced by Bernoulli measures are
particular examples of valuations. Any valuation is characterized by
its values on single pieces, as in~(\ref{eq:15}). 

If $f:\M\to\bbR$ is any valuation, let $h:\C\to\bbR$ be defined by:
\begin{equation}
  \label{eq:16}
\forall \gamma\in\C\quad   h(\gamma)=\sum_{\gamma'\in\C\tq \gamma'\geq \gamma}(-1)^{\m {\gamma'}-\m
  \gamma}f(\gamma')\,. 
\end{equation}

The function $h:\C\to\bbR$ is the \emph{Möbius transform} of~$f$, a
particular instance of the general notion of M\"obius transform in the
sense of Rota~\cite{rota64,stanley86}. For a given valuation
$f:\M\to\bbR$, there exists a Bernoulli measure on the boundary that
induces $f$ if and only the Möbius transform $h$ of $f$ satisfies the
following two conditions:
\begin{align}
\label{eq:18}
  (a)\ h(0)&=0\,;
&(b)\ \forall \gamma\in\Cnv\quad h(\gamma)>0\,.
\end{align}

Note that Condition $(a)$ is a polynomial condition in the
characteristic numbers, and Condition~$(b)$ corresponds to a finite
series of polynomial inequalities. For instance, for the heap monoid
$\T$ on three generators $\T=\langle a,b,c\tq ab=ba\rangle$, we
obtain:
\begin{align}
\label{eq:20}
(a)\   1-p_a-p_b-p_c+p_ap_b&=0\,,&(b)\
\begin{cases}
\begin{aligned}
 h(a)>0&\iff p_a(1-p_b)>0\\  
h(b)>0&\iff p_b(1-p_a)>0\\
h(c)>0&\iff p_c>0\\
h(ab)>0&\iff p_ap_b>0
\end{aligned}
\end{cases}
\end{align}

Returning to the study of a general heap monoid, and when considering
the case where all coefficients $p_a$ are equal, say to~$p$, then both
conditions in~(\ref{eq:18}) reduce to the following: $p$~is the (known
to be unique~\cite{goldwurm00,krob03,csikvari13}) root of smallest
modulus of the \emph{Möbius polynomial of the heap monoid\/~$\M$},
defined by:
\begin{equation*}
  \label{eq:19}
  \mu_\M(X)=\sum_{c\in\C}(-1)^{\m c}X^{\m c}\,.
\end{equation*}
The associated Bernoulli measure is then called the \emph{uniform
  measure on the boundary}.

For the running example $\T$, one has $\mu_\T(X)=1-3X+X^2$\,, and the
uniform measure is given by $\pr(\up x)=p^{\m x}$ with
$p=(3-\sqrt5)/2$\,. 

\medskip In the remaining of this subsection, we characterize the
process of cliques that compose a random infinite heap under a
Bernoulli measure.

For each integer $n\geq1$, the mapping which associates to an infinite
heap $\xi$ the $n^{\text{th}}$~clique $\gamma_n$ such that
$\xi=(\gamma_1,\gamma_2,\ldots)$ is measurable, and defines thus a
random variable $C_n:\BM\to\Cnv$\,. Furthermore, the sequence
$(C_n)_{n\geq1}$ is a \emph{time homogeneous and ergodic Markov
  chain}, of which:
\begin{enumerate}
\item The initial distribution is given by the
  restriction~$h\rest\Cnv$\,, where $h$ is the M\"obius transform
  defined in~(\ref{eq:16}).
\item For each non empty clique $\gamma\in\Cnv$, put:
  \begin{gather}
\label{eq:30}
g(\gamma)=\sum_{\gamma'\in\Cnv\tq\gamma\to \gamma'}h(\gamma')\,.
  \end{gather}
  Then the transition matrix $P=(P_{\gamma,\gamma'})_{(\gamma,\gamma')\in\Cnv}$ of the
  chain is given by:
  \begin{equation}
    \label{eq:21}
    P_{\gamma,\gamma'}=
    \begin{cases}
      0,&\text{if $(\gamma\to \gamma')$ does \emph{not} hold,}\\
h(\gamma')/g(\gamma),&\text{if $(\gamma\to \gamma')$ holds.}
    \end{cases}
  \end{equation}
\end{enumerate}

In general, the initial measure $h\rest\Cnv$ does \emph{not} coincide
with the stationary measure of the chain of cliques.

\section{Asynchronous Stopping Times\\ and the Strong Bernoulli
  Property}
\label{sec:asynchr-stopp-times}

In this section we introduce asynchronous stopping times and their
associated shift operators. They will be our basic tools to formulate
and prove the Law of large numbers in the subsequent sections.

\subsection{Definition and Examples}
\label{sec:defin-exampl-asynchr}

Intuitively, an asynchronous stopping time is a way to select
sub-heaps from infinite heaps, such that for each infinite heap one
can decide at each ``time instant'' whether the sub-heap in question
has already been reached or not. The formal definition follows.

\begin{definition}
\label{def:3}
  An \emph{asynchronous stopping time}, or \ast\ for short, is a
  mapping $V:\BM\to\Mbar$, which we sometimes denote by $\xi\mapsto
  \xi_V$\,, such that:
  \begin{enumerate}
  \item\label{item:5} $\xi_V\leq\xi$ for all $\xi\in\BM$;
  \item\label{item:7} $\forall\xi,\xi'\in\BM\quad (\m{\xi_V}<\infty\wedge\xi_V\leq\xi')
    \implies\xi'_V=\xi_V$\,.
  \end{enumerate}
We say that $V$ is \emph{\pas\ finite} whenever $\xi_V\in\M$ for \pas\
every $\xi\in\BM$. 
\end{definition}

In the above definition, we think of $\xi_V$ as ``$\xi$~cut
at~$V$''. $V=0$ is a first, trivial example of \ast. Actually, the 
property~\ref{item:7} of the above definition implies that if
$\xi_V=0$ for some $\xi\in\BM$, then $V=0$ on~$\BM$. Another simple
example is the following. Let $x\in\M$ be a fixed heap. Define
$V_x:\BM\to\Mbar$ by:
\begin{align*}
  V_x(\xi)=
  \begin{cases}
x\,,&\text{if $x\leq \xi$\,,}\\
\xi\,,&\text{otherwise.}    
  \end{cases}
\end{align*}
Then it is easy to see that $V_x$ is an \ast. We recover the previous
example by setting $x=0$.

Consider the sequence of random cliques $(C_k)_{k\geq1}$ associated
with infinite heaps, and define for each integer $k\geq1$ the random
heap $Y_k=C_1\cdot\ldots\cdot C_k$\,. Then, by construction,
$Y_k\leq\xi$ holds; however, the mapping $Y_k$ is \emph{not} an \ast,
except if $\M$ is the free monoid, which corresponds to the empty
independence relation $I$ on~$\Sigma$. 

We will frequently use the following remark, which is a direct
consequence of the definition: let $V:\BM\to\Mbar$ be an \ast, and let
$\V$ be the set of finite values assumed by~$V$. Then holds:
\begin{gather*}
  \forall x\in\V\qquad\{V=x\}=\up x\,.
\end{gather*}

The following proposition provides less trivial examples of \ast\ that
we will use throughout the rest of the paper. Let us first introduce a
new notation. If $a\in\Sigma$ is a piece, and $x\in\M$ is a heap, the
number of occurrences of $a$ in a representative word of $x$ does not
depend on the representative word, and is thus attached to the
heap~$x$. We denote it by~$|x|_a$\,.

\begin{proposition}
  \label{prop:2}
The two mappings $\BM\to\Mbar,\ \xi\mapsto\xi_V$ described below
define asynchronous stopping times:
\begin{enumerate}
\item\label{item:8} For $\xi=(\gamma_1,\gamma_2,\ldots)$, let $R_\xi=\{k\geq1\tq
  \text{$\gamma_k$ is maximal in $\C$}\}$, and put:
  \begin{align}
  \label{eq:23}
    \xi_V=\begin{cases}
  \gamma_1\cdot\ldots\cdot \gamma_n\text{ with $n=\min R_\xi$}\,,&\text{if
    $R_\xi\neq\emptyset$\,,}\\
\xi\,,&\text{otherwise.}
  \end{cases}
  \end{align}
\item\label{item:9} Let $a\in\Sigma$ be some fixed piece. For
  $\xi\in\BM$, put: 
  \begin{align}
    \label{eq:24}
\xi_V&=\bigwedge H_a(\xi)\,,&   H_a(\xi)&=\{x\in L(\xi)\cap\M\tq\ |x|_a>0\} \,,
  \end{align}
  where $L(\xi)=\{x\in\Mbar\tq x\leq\xi\}$ is the complete lattice
  defined in\/~\eqref{eq:25}, and where the greatest lower bound defining\/
  $\xi_V$ is taken in~$L(\xi)$. It is called the \emph{first hitting time
    of~$a$}.
\end{enumerate}
\end{proposition}

For the first hitting time of~$a$, since the greatest lower
bound in~(\ref{eq:24}) is taken in the complete lattice~$L(\xi)$, if
$H_a(\xi)=\emptyset$ then $\xi_V=\max L(\xi)=\xi$\,. 

\begin{proof}
  \ref{item:8}.\quad
The condition $\xi_V\leq\xi$ is obvious on~(\ref{eq:23}). Hence let
$\xi,\xi'\in\BM$ such that $\xi_V\in\M$ and $\xi_V\leq\xi'$\,. We need
to show that $\xi'_V=\xi_V$\,. For this, let $\xi=(\gamma_1,\gamma_2,\ldots)$
and $\xi'=(\gamma'_1,\gamma'_2,\ldots)$, and let $n=\min R_\xi$ and
$u=\gamma_1\cdot\ldots\cdot \gamma_n$\,. By hypothesis, we have $u\leq\xi'$ and
thus $u\leq \gamma'_1\cdot\ldots\cdot \gamma'_n$\,, by~(\ref{eq:22}).

It follows from \cite[Lemma~8.1]{abbesmair14} that the sequences
$(\gamma_i)_{1\leq i\leq n}$ and $(\gamma'_i)_{1\leq i\leq n}$ are
related as follows: for each integer $i\in\{1,\ldots,n\}$, there
exists a clique $\delta_i$ such that $\delta_i\wedge
\gamma_i=0,\ldots,\delta_i\wedge\gamma_n=0$\,, and
$\gamma'_i=\gamma_i\cdot\delta_i$\,. But $\gamma_n$ is maximal,
therefore $\delta_i=0$ for all $i\in\{1,\ldots,n\}$, and thus
$\gamma'_i=\gamma_i$\,. From $\gamma'_n=\gamma_n$ follows at once that
$\min R_{\xi'}\leq n$. And since $\gamma'_i=\gamma_i$ for all $i<n$,
no clique $\gamma'_i$ is maximal in~$\C$, otherwise it would
contradict the definition of~$\xi_V$\,. Hence finally $\min
R_{\xi'}=n$\,, from which follows $\xi'_V=\gamma'_1\cdot\ldots\cdot
\gamma'_n=\gamma_1\cdot\ldots\cdot \gamma_n=\xi_V$\,.

\medskip \ref{item:9}.\quad Again, it is obvious on~(\ref{eq:24}) that
$\xi_V\leq\xi$. Let $\xi,\xi'\in\BM$ be such that $\xi_V\in\M$ and
$\xi_V\leq\xi'$\,.  Then $H_a(\xi)\neq\emptyset$, and we observe that
$H_a(\xi)$ actually has a minimum, $\xi_V=\min H_a(\xi)$. This is best
seen with the resource interpretation of heap monoids introduced
in~\cite{cori85}. 


It follows that $\xi_V\in H_a(\xi')$ and thus
$H_a(\xi')\neq\emptyset$. Henceforth, as above, we deduce $\xi'_V=\min
H_a(\xi')$, and $\xi'_V\leq\xi_V$\,. It follows that $\xi'_V\in
H_a(\xi)$ and thus $\xi_V\leq\xi'_V$ and finally $\xi_V=\xi'_V$\,.
\end{proof}

\subsection{Action of the Monoid on its Boundary. Shift Operators}
\label{sec:action-monoid-its}

In order to define the shift operator associated to an \ast, we first
describe the natural left action of a heap monoid on its boundary.

For $x\in\M$ and $\xi\in\BM$ an infinite heap, the visually intuitive
operation of piling up $\xi$ upon $x$ should yield an infinite
heap. However, some pieces in the first layers of $\xi$ might fall off
and fill up empty slots in~$x$. Hence the \CF\ decomposition of
$x\cdot\xi$ cannot be defined as the mere concatenation of the \CF\
decompositions of $x$ and~$\xi$.

The proper definition of the concatenation $x\cdot\xi$ is as
follows. Let $\xi=(\gamma_1,\gamma_2,\ldots)$.  The sequence of heaps
$(x\cdot \gamma_1\cdot\ldots\cdot \gamma_n)_{n\geq1}$ is obviously non
decreasing. According to point~\ref{item:13} of
Proposition~\ref{prop:1}, we may thus consider:
\begin{equation*}
  x\cdot\xi=\bigvee_{n\geq1}(x\cdot \gamma_1\cdot\ldots\cdot
  \gamma_n)\,,\quad\text{which exists in $\BM$\,,}
\end{equation*}
and then we have:
\begin{gather*}
  \forall x,y\in\M\quad\forall\xi\in\BM\quad
  x\cdot(y\cdot\xi)=(x\cdot y)\cdot\xi\,.
\end{gather*}

It is then routine to check that, for each $x\in\M$, the mapping:
\begin{equation*}
  \label{eq:26}
  \Phi_x:\BM\to\up x\,,\quad \xi\mapsto x\cdot\xi\,,
\end{equation*}
is a bijection. Therefore, we extend the notation $y-x$, licit for
$x,y\in\M$ with $x\leq y$, by allowing $y$ to range over~$\BM$, as
follows:
\begin{equation*}
  \label{eq:28}
  \forall x\in\M\quad\forall\xi\in\up x\quad\xi-x=\Phi_x^{-1}(\xi)\,.
\end{equation*}

Hence $\zeta=\xi-x$ denotes the tail of $\xi$ ``after''~$x$, for
$\xi\geq x$. It is characterized by the property $x\cdot\zeta=\xi$\,,
and this allows us to introduce the following definition.

\begin{definition}
  \label{def:1}
  Let $V:\BM\to\Mbar$ be an \ast. The \emph{shift operator associated
    to~$V$} is the mapping\/ $\theta_V:\BM\to\BM$, which is partially
  defined by:
\begin{equation*}
  \label{eq:29}
  \forall\xi\in\BM\quad \xi_V\in\M\implies\theta_V(\xi)=\xi-\xi_V\,.
\end{equation*}
The domain of definition of\/ $\theta_V$ is\/ $\{\m{\xi_V}<\infty\}$\,.
\end{definition}

\subsection{The Strong Bernoulli Property}
\label{sec:strong-bern-prop}

The Strong Bernoulli property has with respect to the definition of
Bernoulli measures, the same relationship than the Strong Markov
property with respect to the mere definition of Markov chains. Its
formulation is also similar (see for instance~\cite{revuz75}). In
particular, it involves a \slgb\ associated with an \ast, defined as
follows.

\begin{definition}
  \label{def:2}
  Let $V:\BM\to\Mbar$ be an\/ \ast, and let $\V$ be the collection of
  finite values assumed by~$V$. We define the \slgb\ $\FFF_V$ as
\begin{gather*}
\FFF_V=\sigma\langle \up x\tq x\in\V\rangle\,.
\end{gather*}
\end{definition}

With the above definition, we have the following result.

\begin{theorem}
  \label{thr:1}
  \textbf{{\sffamily(Strong Bernoulli Property)}}\quad Let\/ $\pr$ be
  a Bernoulli measure on~$\BM$\,, let $V:\BM\to\Mbar$ be an \ast\ and
  let \mbox{$\psi:\BM\to\bbR$} be a $\FFF$-measurable function, either
  non negative or\/ $\pr$-integrable. Extend $\psi\circ\theta_V$\,,
  which is only defined on $\{|\xi_V|<\infty\}$, by
  $\psi\circ\theta_V=0$ on $\{|\xi_V|=\infty\}$.  Then:
  \begin{gather}
\label{eq:40}
    \esp(\psi\circ\theta_V|\FFF_V)=\esp(\psi)\,,\quad\text{\pas\ on
      $\{|\xi_V|<\infty\}$\,,}
  \end{gather}
  denoting by $\esp(\cdot)$ the expectation with respect to\/~$\pr$,
  and by $\esp(\cdot|\FFF_V)$ the conditional expectation with respect
  to\/~$\pr$ and to the \slgb\/~$\FFF_V$\,.
\end{theorem}

If $V:\BM\to\M$ is \pas\ finite, then the Strong Bernoully Property
writes as:
\begin{gather*}
\pas\quad  \esp(\psi\circ\theta_V|\FFF_V)=\esp(\psi)\,.
\end{gather*}

In the general case, we may still have an equality valid
$\pr$-almost surely by multiplying both members of~(\ref{eq:40}) by
the characteristic function~$\un_{\{V\in\M\}}$\,, which is
$\FFF_V$-measurable, as follows:
\begin{gather}
\pas\quad \esp(\un_{\{V\in\M\}}\psi\circ\theta_V|\FFF_V)=\un_{\{V\in\M\}}\esp(\psi)\,.
\end{gather}

\begin{proof}
It is enough to show the result for $\psi$ of the form $\psi=\un_{\up
  y}$ for some heap $y\in\M$. For such a function~$\psi$, let
$Z=\esp(\psi\circ\theta_V|\FFF_V)$\,. 

Let $\V$ be the set of finite values assumed by~$V$. We note that the
cylinders~$\up x$\,, for $x$ ranging over~$\V$, are pairwise disjoint
since $V(\xi)=x$ on $\up x$ for $x\in\V$. Hence $\FFF_V$ is
atomic. Therefore, if $\hat Z:\V\to\bbR$ denotes the function defined
by:
 \begin{gather*}
   \forall x\in\V\qquad\hat Z(x)=\esp(\psi\circ\theta_V|V=x)\,,
 \end{gather*}
 then a version  of $Z$ is given by:
\begin{gather*}
  Z(\xi)=
  \begin{cases}
0\,,&\text{if $V(\xi)=\xi$\,,} \\
\hat Z(x)\,,&\text{if $V(\xi)=x$ with $x\in\M$\,.}
  \end{cases}
\end{gather*}

For $x\in\V$ and for $\xi\geq x$, one has
\begin{gather*}
\psi\circ\theta_V(\xi)=\psi(\xi-x)=\un_{\up y}(\xi-x)=\un_{\up(x\cdot y)}(\xi)\,.
\end{gather*}
And since $\{V=x\}=\up x$ for $x\in\V$, this yields:
\begin{align*}
  \hat Z(x)=\esp(\psi\circ\theta_V|\up x)=\frac1{\pr(\up
    x)}\pr(\up(x\cdot y))=\pr(\up y)=\esp\psi\,,
\end{align*}
by the multiplicativity property of~$\pr$\,. The proof is complete. 
\end{proof}

\section{Iterating Asynchronous Stopping Times}
\label{sec:iter-asynchr-stopp}

This section studies the iteration of asynchronous stopping times,
defined in a very similar way as the iteration of classical stopping
times for standard probabilistic processes; see for
instance~\cite{revuz75}. Properly dealing with iterated \ast\ is a
typical example of use of the Strong Bernoulli property, as in
Proposition~\ref{prop:3} below.

\subsection{Iterated Stopping Times}
\label{sec:iter-stopp-times}

\begin{proposition}
\label{prop:4}
Let $V:\BM\to\Mbar$ be an\/ \ast. Let $V_0=0$\,, and
define the mappings $V_n:\BM\to\Mbar$ by induction as follows:
\begin{align*}
\forall\xi\in\BM\qquad V_{n+1}(\xi)=
\begin{cases}
  \xi\,,&\text{if $V_n(\xi)\in\BM$}\\
V_n(\xi)\cdot V\bigl(\xi-V_n(\xi)\bigr)\,,&\text{if $V_n(\xi)\in\M$}
\end{cases}
\end{align*}
Then $(V_n)_{n\geq0}$ is a sequence of \ast.
\end{proposition}

\begin{proof}
  The proof is by induction on the integer $n\geq0$. 

  The case $n=0$ is trivial. Hence, for $n\geq1$, and assuming that
  $V_{n-1}$ is an \ast, let $\xi,\xi'\in\BM$ be such that:
\begin{gather*}
  \begin{aligned}
    V_n(\xi)&\in\M\,,& V_n(\xi)&\leq\xi'\,.
  \end{aligned}
\end{gather*}

It implies in particular that $V_{n-1}(\xi)\in\M$ and
  $V_{n-1}(\xi)\leq\xi'$\,, from which follows by the induction
  hypothesis that $V_{n-1}(\xi')=V_{n-1}(\xi)$\,. Putting
  $x=V_{n-1}(\xi)=V_{n-1}(\xi')$ on the one hand, there are thus two
  infinite heaps $\zeta$ and $\zeta'$ such that $\xi=x\cdot\zeta$ and
  $\xi'=x\cdot\zeta'$\,. Putting $y=V(\xi-V_{n-1}(\xi))$ on the other
  hand, the assumption $V_n(\xi)\leq\xi'$ writes as:
  $ x\cdot y\leq x\cdot\zeta'$, which implies $y\leq\zeta'$ by
  cancellativity of the monoid. But since $V$ is an \ast, this implies
  in turn $V(\zeta')=y$, and finally, by definition of~$V_n$\,:
  \begin{gather*}
V_n(\xi')=V_{n-1}(\xi')\cdot V(\xi'-V_{n-1}(\xi'))=x\cdot
V(\zeta')=x\cdot y=V_n(\xi)\,.
  \end{gather*}
This shows that $V_n$ is an \ast, completing the induction.
\end{proof}

\begin{definition}
  \label{def:5}
Let $V:\BM\to\Mbar$ be an\/ \ast. The sequence $(V_n)_{n\geq0}$ of\/ \ast\
defined as in Proposition\/~{\normalfont\ref{prop:3}} is called the
\emph{iterated sequence of stopping times} associated with~$V$.
\end{definition}

\begin{proposition}
  \label{prop:3}
  Let\/ $\pr$ be a Bernoulli measure equipping the boundary~$\BM$.
  Let $(V_n)_{n\geq0}$ be the iterated sequence of stopping times
  associated with an \ast\/ $V:\BM\to\Mbar$ which we assume to be\/
  \pas\ finite. Let also $(\Delta_n)_{n\geq1}$ be the sequence of
  increments:
\begin{align*}
  \forall n\geq0\qquad \Delta_{n+1}&=V\circ\theta_{V_{n}}\,,&
  V_{n+1}&=V_n\cdot\Delta_{n+1}\,.
\end{align*}
Then $(\Delta_n)_{n\geq1}$ is an \iid\ sequence of random variables
with values in\/~$\M$, with the same distribution as~$V$.
\end{proposition}

\begin{proof}
  We first show that $V_n\in\M$ for all integers $n\geq1$ and
  $\pr$-almost surely. For this, we apply the Strong Bernoulli
  property (Theorem~\ref{thr:1}) with \ast\ $V_{n-1}$ and with the
  function $\psi=\un_{\{V\in\M\}}$ to get:
  \begin{gather*}
    \pas\quad
    \esp\bigl(\un_{\{V_{n-1}\in\M\}}\psi\circ\theta_{V_{n-1}}|\FFF_{V_{n-1}})=\un_{\{V_{n-1}\in\M\}}\esp\psi\,.
  \end{gather*}
  But
  $\un_{\{V_n\in\M\}}=\un_{\{V_{n-1}\in\M\}}\psi\circ\theta_{V_{n-1}}$\,,
  and $\esp\psi=\pr(V\in\M)=1$ by hypothesis. Hence the equation above
  writes as:
\begin{gather*}
\pas\quad\esp(\un_{\{V_n\in\M\}}|\FFF_{V_{n-1}})=\un_{\{V_{n-1}\in\M\}}\,.
\end{gather*}
Taking the expectations of both members yields:
$\pr(V_n\in\M)=\pr(V_{n-1}\in\M)$\,. Hence by induction, since
$\pr(V_0\in\M)=1$\,, we deduce that $\pr(V_n\in\M)=1$ for all integers $n\geq1$.

To complete the proof of the proposition, we show that for any non
negative functions $\varphi_1,\ldots,\varphi_n:\M\to\bbR$\,, holds:
\begin{gather}
\label{eq:31}
  \esp\bigl(\varphi_1(\Delta_1)\cdot\ldots\cdot\varphi_n(\Delta_n)\bigr)
=\esp\varphi_1(V)\cdot\ldots\cdot\esp\varphi_n(V)\,.
\end{gather}

The case $n=0$ is trivial. Assume the hypothesis true at rank
$n-1\geq0$. Applying the Strong Bernoulli property
(Theorem~\ref{thr:1}) with the \ast\ $V_{n-1}$ yields, since $V_{n-1}\in\M$
$\pr$-almost surely:
\begin{gather}
  \label{eq:32}
\pas\quad\esp\bigl(\varphi_n(\Delta_n)|\FFF_{V_{n-1}}\bigr)=\esp\varphi_n(V)\,.
\end{gather}

Let $A$ be the left-hand member of~(\ref{eq:31}).  Since
$\Delta_1,\ldots,\Delta_{n-1}$ are $\FFF_{V_n}$-measurable, we compute
as follows, using the standard properties of conditional
expectation~\cite{billingsley95}:
\begin{align*}
  A&=\esp\bigl(\esp(\varphi_1(\Delta_1)\cdot\ldots\cdot\varphi_n(\Delta_n)|\FFF_{V_{n-1}})\bigr)\\
  &=\esp\bigl(\varphi_1(\Delta_1)\cdot\ldots\cdot\varphi_{n-1}(\Delta_{n-1})\cdot\esp(\varphi_n(\Delta_n)|\FFF_{V_{n-1}})\bigr)\\
  &=\esp(\varphi_1(\Delta_1)\cdot\ldots\cdot\varphi_{n-1}(\Delta_{n-1})\bigr)\cdot\esp\varphi_n(V)&\text{by~(\ref{eq:32})}\\
  &=\esp\varphi_1(V)\cdot\ldots\cdot\esp\varphi_n(V)\,,
\end{align*}
the later equality by the induction hypothesis. This
proves~(\ref{eq:31}).
\end{proof}

\subsection{Exhaustive Asynchronous Stopping Times}
\label{sec:exha-asynchr-stopp}

\begin{lemma}
  \label{lem:1}
Let $V$ be an\/ \ast, that we assume to be \pas\ finite. Let
$(V_n)_{n\geq0}$ be the associated sequence of iterated stopping
times. Then the two following properties are equivalent:
\begin{enumerate}[(i)]
\item\label{item:15} $\FFF=\bigvee_{n\geq0}\FFF_{V_n}$\,.
\item\label{item:16} $\xi=\bigvee_{n\geq0}V_n(\xi)$ for \pas\ every $\xi\in\BM$.
\end{enumerate}
\end{lemma}

\begin{proof}
  (\ref{item:15}) implies~(\ref{item:16}).\quad By the Martingale
  convergence theorem \cite[Th.~35.6]{billingsley95}, we have for
  every $x\in\M$:
  \begin{gather}
\label{eq:10}
\pas\quad \lim_{n\to\infty}\esp(\un_{\up x}|\FFF_{V_n})=\un_{\up x}\,.
  \end{gather}

  Let $n\geq0$ be an integer, and let $\V_n$ denote the set of finite
  heaps assumed by~$V_n$\,. Then, by the \ast\ property of~$V_n$\,, we
  have \mbox{$\{V_n=y\}=\up y$} for every $y\in\V_n$\,, and thus, for
  every $x\in\M$ compatible with~$y$\,:
  \begin{align*}
    \pr(\up x\,| V_n=y)&=\pr(\up x\,|\up y)=\pr\bigl(\up(x-(x\wedge
    y))\bigr)&&\text{by~(\ref{eq:9})}
  \end{align*}

  Therefore, by~(\ref{eq:10}), we obtain for every $x\in\M$ and for
  \pas\ every $\xi\in\up x$:
\begin{align*}
\lim_{n\to\infty}  \pr(\up(x-(x\wedge\xi_{V_n}))=1.
\end{align*}

It implies that the sequence $x\wedge\xi_{V_n}$\,, which is eventually
constant since it is non decreasing and bounded by~$x$, eventually
reaches~$x$, since the only heap $y\in\M$ satisfying $\pr(\up y)=1$ is
$y=0$. In other words: $x\leq\xi_{V_n}$ for $n$ large enough. Hence:
$\bigvee_{n\geq0}\xi_{V_n}\geq x$ for every $x\in\M$ and for \pas\
every $\xi\in\up x$. In view of the basis property of~$\M$
(point~\ref{item:17} of Proposition~\ref{prop:1}), it follows that
$\bigvee_{n\geq0}\xi_{V_n}=\xi$ holds $\pr$-almost surely.

(\ref{item:16}) implies~(\ref{item:15}).\quad Let $\FFF'$ be the
\slgb:
\begin{gather*}
  \FFF'=\bigvee_{n\geq1}\FFF_{V_n}\,.
\end{gather*}
To show that $\FFF=\FFF'$\,, it is enough to show that $\up x\in\FFF'$
for every $x\in\M$. But for $x\in\M$, by assumption
$\bigvee_{n\geq1}V_n(\xi)\geq x$ for \pas\ every $\xi\in\up x$. By the
finiteness property of elements of~$\M$ (point~\ref{item:12} of
Proposition~\ref{prop:1}), it implies, for \pas\ every $\xi\in\up x$,
the existence of an integer $n\geq0$ such that $V_n(\xi)\geq
x$. Letting $\V$ denote the set of finite values assumed by any of
the~$V_n$\,, we have thus:
\begin{gather*}
  \up x=\bigcup_{v\in\V\tq v\geq x}\up v\,.
\end{gather*}
Since $\V$ is at most countable, it implies $\up x\in\FFF'$, which was
to be shown.
\end{proof}

\begin{definition}
  \label{def:6}
  A\/ \pas\ finite\/ \ast\ that satisfies any of the
  properties~(\ref{item:15})--(\ref{item:16}) of
  Lemma\/~{\normalfont\ref{lem:1}} is said to be \emph{exhaustive}.
\end{definition}

\subsection{Examples of Exhaustive Asynchronous Stopping Times}
\label{sec:exampl-exha-asynchr}

\begin{proposition}
  \label{prop:5}
  Both examples $V$ of \ast\ defined in
  Proposition\/~{\normalfont\ref{prop:2}} are exhaustive, and satisfy
  furthermore\/ $\esp|V|<\infty$.
\end{proposition}

\begin{proof}
  For both examples, that $|V|<\infty$ \pas\ and also that
  $\esp|V|<\infty$, follow from the two following facts:
  \begin{enumerate}
  \item The Markov chain of cliques $(C_k)_{k\geq1}$ such that
    $\xi=(C_1,C_2,\ldots)$ is irreducible with a finite number of
    states (see~\S~\ref{sec:bernoulli-measures}), and thus is positive
    recurrent.
  \item If $\alpha$ denotes the maximal size of a clique, then
    $|C_1\cdot\ldots\cdot C_k|\leq\alpha k$.
  \end{enumerate}

We now show that both examples are exhaustive. Let $(V_n)_{n\geq0}$
be the associated sequence of iterated stopping times. 

\textit{For $V$ defined in point~\ref{item:8} of
  Proposition~\ref{prop:2}.}\quad Since $V<\infty$ \pas, it follows
from Proposition~\ref{prop:3} that $V_n<\infty$ \pas\ and for all
integers $n\geq0$. Let $\xi=(\gamma_n)_{n\geq1}$ be an infinite
heap. Let $n\geq0$ be an integer, and let $c_1\to\ldots\to c_{k_n}$ be
the \CF\ decomposition of~$V_n(\xi)$. Then, on the one hand, $k_n\geq
n$, and on the other hand, since $c_{k_n}$ is maximal and since
$\xi\geq V_n(\xi)$, it must hold:
\begin{align*}
  \gamma_1&=c_1\,,&\gamma_2&=c_2\,,&&\ldots&\gamma_{k_n}&=c_{k_n}\,.
\end{align*}
Hence, if $\xi'=(\gamma'_n)_{n\geq1}$ denotes:
\begin{gather*}
  \xi'=\bigvee_{n\geq0}V_n(\xi)\,,
\end{gather*}
one has $\gamma'_i=\gamma_i$ for all $i\leq k_n$ and for all
$n\geq0$. And since $k_n\geq n$, it implies $\gamma_i=\gamma'_i$ for
all integers $i\geq1$, and thus $\xi'=\xi$. This proves that $V$ is
exhaustive. 

\textit{For $V$ defined in point~\ref{item:9} of
  Proposition~\ref{prop:2}.}\quad Let $V$ be the first hitting time of
$a\in\Sigma$. With the same notations as above, let us first show the
following claim:
\begin{enumerate}
\item[$(\Diamond)$] For every $b\in\Sigma$:\quad $\pr(V\geq b)>0$.
\item[$(\Diamond\Diamond)$] For every $b\in\Sigma$, and $\pr$-almost surely:\quad
  $\xi\geq b\implies \xi'\geq b$.
\end{enumerate}

\textit{Proof of\/~$(\Diamond)$.}\quad Since the dependence relation
$D=(\Sigma\times\Sigma)\setminus I$ is assumed to make the graph
$(\Sigma,D)$ connected, we pick a sequence $a_1,\ldots,a_j$ of
pairwise distinct pieces such that $a_1=b$, $a_j=a$ and
$(a_i,a_{i+1})\in D$ for all $i\in\{1,\ldots,j-1\}$\,. Put
$x=a_1\cdot\ldots\cdot a_j$\,. Then it is clear that $V(\xi)=x$ for
every $\xi\geq x$. Hence $\pr(V=x)=\pr(\up x)>0$. Since $b\leq x$, it
follows that $\pr(V\geq b)\geq\pr(\up x)>0$.

\textit{Proof of\/~$(\Diamond\Diamond)$.}\quad Let
$(\Delta_n)_{n\geq1}$ be the sequence of increments, such that
$V_{n+1}=V_n\cdot\Delta_{n+1}$\,. Then $(\Delta_n)_{n\geq1}$ being \iid\
with the same law as $V$ according to Proposition~\ref{prop:3}, and
since $\pr(V\geq b)>0$, it follows from Borel-Cantelli Lemma that
there exists at least an integer $n\geq1$ such that $\Delta_n\geq b$,
for \pas\ every $\xi\in\BM$. For \pas\ every $\xi\geq b$, let $n$ be
the smallest such integer. Then the heap
$\Delta_1\cdot\ldots\cdot\Delta_{n-1}$ does not contain any occurrence
of $b$ on the one hand, and is compatible with $b$ on the other
hand. That implies that $b$ commutes with all pieces of
$\Delta_1\cdot\ldots\cdot\Delta_{n-1}$\,. Therefore, it follows that
$b\leq\Delta_1\cdot\ldots\cdot\Delta_n\leq\xi'$\,. The claim
$(\Diamond\Diamond)$ is proved.

Now, to prove that $V$ is exhaustive, let $x\in\M$ be a heap. We show
that, \pas, $\xi\geq x\implies\xi'\geq x$, which will complete the
proof \textit{via} the basis property of~$\M$ (point~\ref{item:17} of
Proposition~\ref{prop:1}). Putting $y=\xi'\wedge x$, and assuming
$\xi\geq x$, we prove that $y=x$ holds $\pr$-almost surely. Assume
$y\neq x$. Since $y\leq x$, there is thus a piece $b\in\Sigma$ such
that $y\cdot b\leq x$ holds and $(\xi'-y)\geq b$ does not hold. Let
$N$ be the smallest integer such that $V_N(\xi)\wedge x=y$\,; such an
integer exists, by the finiteness property of~$\BM$
(point~\ref{item:12} of Proposition~\ref{prop:1}). Let
$z=V_N(\xi)$. Then it follows from the definition of the
sequence $(V_n)_{n\geq0}$ that holds:
\begin{gather*}
  \forall n\geq N\qquad V_n(\xi)=z\cdot V_{N-n}(\xi-z)\,.
\end{gather*}

According to the property $(\Diamond\Diamond)$, for \pas\ every $\xi$
such that $\xi-z\geq b$, there exists an integer $k\geq0$ such that
$V_k(\xi-z)\geq b$. But then, $V_{N+k}(\xi)\geq y\cdot b$, and thus
$\xi'\wedge x\geq y\cdot b$, contradicting the definition of~$y$. It
follows that $y\neq x$ can only occur with probability~$0$, which was
to be proved.
\end{proof}

\section{The Cut-Invariant Law of Large Numbers}
\label{sec:cut-invariant-law}

\subsection{Statement of the Law of Large Numbers}
\label{sec:statement-strong-law}

We first define ergodic sums and ergodic means associated with an
\ast\ and with a cost function. The setting of the section is the same
as previously: a heap monoid $\M=\M(\Sigma,I)$ together with a
Bernoulli measure $\pr$ on $(\BM,\FFF)$.

If $\varphi:\Sigma\to\bbR$ is a function, seen as a cost function, it
is clear that $\varphi$ has a unique extension on $\M$ which is
additive; we denote this extension by $\scal\varphi\cdot$. Hence, if
the $\Sigma$-word $x_1\ldots x_n$ is a representative of a heap~$x$, then:
\begin{gather*}
  \scal\varphi x=\varphi(x_1)+\ldots+\varphi(x_n)\,.
\end{gather*}

In particular, if $1$ denotes the constant function, equal to $1$
on~$\Sigma$, one has: $\scal1x=|x|$ for every $x\in\M$.

\begin{definition}
  \label{def:4}
  Let $V:\BM\to\Mbar$ be an\/ \ast, which we assume to be\/ \pas\
  finite, and let $\varphi:\Sigma\to\bbR$ be a cost function. Consider
  the sequence of iterated stopping times $(V_n)_{n\geq0}$ associated
  with~$V$. The\/ \emph{ergodic sums} associated with $V$ are the
 random variables in the sequence $(S_{V,n}\varphi)_{n\geq0}$ defined\/
  \pas\ by:
\begin{gather*}
\forall n\geq0\qquad  S_{V,n}\varphi=\scal\varphi{\xi_{V_n}}\,.
\end{gather*}

The \emph{ergodic means} associated with $V$ are the random variables
in the sequence $(M_{V,n})_{n\geq1}$ defined\/ \pas\ by:
\begin{gather*}
  \forall n\geq0\qquad M_{V,n}\varphi
=\frac{S_{V,n}\varphi}{S_{V,n}1}=\frac{\scal\varphi{\xi_{V_n}}}{|\xi_{V_n}|}\,.
\end{gather*}
\end{definition}

\begin{theorem}
  \label{thr:2}
  Let $\M(\Sigma,I)$ be a trace monoid, equipped with a Bernoulli
  measure $\pr$ on\/ $(\BM,\FFF)$ and with a cost function
  $\varphi:\Sigma\to\bbR$.

  Then, for every exhaustive asynchronous stopping time\/
  \mbox{$V:\BM\to\Mbar$} such that\/ \mbox{$\esp|V|<\infty$} holds,
  the ergodic means $(M_{V,n})_{n\geq1}$ converge \pas\ toward a
  constant. Furthermore, this constant does not depend on the choice
  of the exhaustive \ast\ $V$ such that \mbox{$\esp|V|<\infty$}.
\end{theorem}

Before we proceed with the proof of Theorem~\ref{thr:2}, we state a
corollary which provides a practical way of computing the limit of
ergodic means. 

\begin{corollary}
  \label{cor:1}
  Let $\varphi:\Sigma\to\bbR$ be a cost function. Let $\pi$ be the
  invariant measure of the Markov chain of cliques associated with a
  Bernoulli measure\/ $\pr$ on~$\BM$. Then the limit $M\varphi$ of the
  ergodic means~$M_{V,n}\varphi$\,, for any exhaustive \ast\ $V$ such
  that\/ $\esp|V|<\infty$ holds, is given by:
\begin{gather}
\label{eq:11}
  M\varphi=\Bigl(\sum_{\gamma\in\Cstar}\pi(\gamma)|\gamma|
\Bigr)^{-1}
\sum_{\gamma\in\Cstar}\pi(\gamma)\scal\varphi\gamma\,.
\end{gather}
\end{corollary}

\begin{proof}
  According to Theorem~\ref{thr:2}, to compute the value $M\varphi$,
  we may choose any \ast\ of finite length in average. According to
  Proposition~\ref{prop:2}, the \ast\ $V:\BM\to\Mbar$ defined in
  point~\ref{item:8} of Proposition~\ref{prop:2} is eligible. But then
  the ergodic means are given by:
  \begin{align*}
    M_{V,n}\varphi&=\frac{K_n}{|C_1|+\ldots+|C_{K_n}|}\cdot \frac{\varphi(C_1)+\ldots+\varphi(C_{K_n})}{K_n}
  \end{align*}
  where $(C_k)_{k\geq1}$ is the Markov chain of cliques associated
  with infinite heaps, and for some integers $K_n$ such that
  $\lim_{n\to\infty} K_n=\infty$\,. The equality~(\ref{eq:11}) follows
  then from the Law of large numbers \cite{chung60} for the ergodic Markov
  chain~$(C_k)_{k\geq1}$\,.
\end{proof}

\subsection{Direct Computation for the Introductory Probabilistic Protocol}
\label{sec:direct-comp-intr}

We have computed in \S~\ref{sec:cut-invariance-an} the asymptotic
density of pieces for the heap monoid $\T=\langle a,b,c\;|\;
ab=ba\rangle$ by computing ergodic means associated either with the
first hitting time of $c$ or with the first hitting time of~$a$. The
fact that the results coincide can be seen as an instance of
Theorem~\ref{thr:2}. Corollary~\ref{cor:1} provides a direct way of
computing the limit density vector, without having to describe an
infinite set of heaps as we did in~\S~\ref{sec:cut-invariance-an},
which would become much less tractable for a general heap monoid. Let
us check that we recover the same values for the density vector
$\gamma=\begin{pmatrix}\gamma_a&\gamma_b&\gamma_c
\end{pmatrix}$\,.

We have already obtained in~(\ref{eq:27}) the values
of the characteristic numbers of the associated Bernoulli measure:
$f(a)=1-\lambda$, $f(b)=1-\lambda'$, $f(c)=\lambda\lambda'$. Let us
use the short notations $a,b,c$ for $f(a), f(b), f(c)$. The M\"obius
transform is then the following vector, indexed by cliques $a,b,c,ab$
in this order:
\begin{gather*}
  h=
  \begin{pmatrix}
    a(1-b)&b(1-a)&c&ab
  \end{pmatrix}
\end{gather*}

Using the equality $1-a-b-c+ab=0$, and according to the results
recalled in~\S~\ref{sec:bernoulli-measures}, the transition matrix of
the chain of cliques is given by:
\begin{gather*}
  P=
  \begin{pmatrix}
a&   0&\frac c{1-b}&0\\
0&b&\frac c{1-a}&0\\
a(1-b)&b(1-a)&c&ab\\
a(1-b)&b(1-a)&c&ab
  \end{pmatrix}=
  \begin{pmatrix}
1-\lambda&0&\lambda&0\\
0&1-\lambda'&\lambda'&0\\
\lambda'(1-\lambda)&\lambda(1-\lambda')&\lambda\lambda'&(1-\lambda)(1-\lambda')\\    
\lambda'(1-\lambda)&\lambda(1-\lambda')&\lambda\lambda'&(1-\lambda)(1-\lambda')
  \end{pmatrix}
\end{gather*}

Direct computations give the left invariant probability
vector~$\pi$ of~$P$:
\begin{gather*}
  \begin{pmatrix}
    \pi_a\\\pi_b\\\pi_c\\\pi_{ab}
  \end{pmatrix}
=\frac1{\lambda^2+\lambda'^2-\lambda\lambda'(\lambda+\lambda'-1)}
  \begin{pmatrix}
    (1-\lambda)\lambda'^2\\(1-\lambda')\lambda^2\\
\lambda\lambda'(\lambda+\lambda'-\lambda\lambda')\\
 \lambda\lambda'(1-\lambda)(1-\lambda')
  \end{pmatrix}
\end{gather*}

Using the notion of limit for ergodic means, the density vector
defined in~\S~\ref{sec:cut-invariance-an} is $\gamma=
\begin{pmatrix} \gamma_a&\gamma_b&\gamma_c
\end{pmatrix}=
\begin{pmatrix}M\un_{\{a\}}&M\un_{\{b\}}&M\un_{\{c\}}
\end{pmatrix}$\,, which yields, according to the result of
Corollary~\ref{cor:1}:
\begin{align*}
  \begin{pmatrix}
    \gamma_a\\\gamma_b\\\gamma_c
  \end{pmatrix}=\frac1{\pi_a+\pi_b+\pi_c+2\pi_{ab}}
\begin{pmatrix}
\pi_a+\pi_{ab}
\\\pi_b+\pi_{ab}\\
\pi_c
\end{pmatrix}=\frac1{\lambda+\lambda'-\lambda\lambda'}
\begin{pmatrix}
\lambda'(1-\lambda)\\
\lambda(1-\lambda')\\
\lambda\lambda'
\end{pmatrix}
\end{align*}
As expected, we recover the values found
in~\S~\ref{sec:cut-invariance-an}.

\subsection{Proof of Theorem~\ref{thr:2}}
\label{sec:proof-theorem}

The proof is divided into two parts, each one gathered in a
subsection: first, the proof of convergence of the ergodic
means~(\S~\ref{sec:proof-convergence}); and second, the proof that the
limit does not depend on the choice of the \ast\ $V:\BM\to\Mbar$,
provided that $\esp|V|<\infty$ holds~(\S~\ref{sec:uniqueness-limit}).

\subsubsection{Convergence of Ergodic Means}
\label{sec:proof-convergence}

Using the notations introduced in Theorem~\ref{thr:2}, let
$(\Delta_n)_{n\geq1}$ be the sequence of increments associated with
the sequence~$(V_n)_{n\geq1}$\,. The increments are defined as in
Proposition~\ref{prop:3}. Then we have:
\begin{align*}
  M_{V,n}\varphi&=\frac{\scal\varphi{\Delta_1\cdot\ldots\cdot\Delta_n}}{\scal1{\Delta_1\cdot\ldots\cdot\Delta_n}}=\frac
  n{\scal1{\Delta_1}+\ldots+\scal1{\Delta_n}}\cdot
  \frac{\scal\varphi{\Delta_1}+\ldots+\scal\varphi{\Delta_n}}n
\end{align*}

Let $M=\max|\varphi|$\,. Then the assumption
$\esp|\xi_V|<\infty$ implies:
\begin{gather*}
  \esp|\scal\varphi{\xi_V}|\leq M\esp|\xi_V|<\infty\,.
\end{gather*}

Since $(\Delta_n)_{n\geq1}$ is \iid\ according to
Proposition~\ref{prop:3}, each $\Delta_i$ being distributed according
to~$\xi_V$, the Strong law of large numbers for \iid\ sequences
implies the \pas\ convergence:
\begin{align}
\label{eq:34}
\lim_{n\to\infty}\frac{\scal\varphi{\Delta_1}+\ldots+\scal\varphi{\Delta_n}}n
&=\esp\scal\varphi{\xi_V}\,,
&\lim_{n\to\infty}\frac{\scal1{\Delta_1}+\ldots+\scal1{\Delta_n}}n
=\esp|\xi_V|\,.
\end{align}

It follows in particular from $V$ being exhaustive that
\mbox{$\esp|\xi_V|>0$}; otherwise, we would have $\xi_V=0$, \pas, and
thus $\xi_{V_n}=0$, \pas\ and for all $n\geq0$, contradicting the
\pas\ equality $\bigvee_{n\geq0}\xi_{V_n}=\xi$ stated in
Lemma~\ref{lem:1}. Hence, from~(\ref{eq:34}), we deduce the \pas\
convergence:
\begin{gather*}
  \lim_{n\to\infty}M_{V,n}\varphi=\frac{\esp\scal\varphi{\xi_V}}{\esp|\xi_V|}\,.
\end{gather*}

\subsubsection{Uniqueness of the Limit}
\label{sec:uniqueness-limit}

We start with a couple of lemmas. 

\begin{lemma}
  \label{lem:4}
Let $f:\M\to\bbR$ be the valuation defined by $f(x)=\pr(\up x)$ for
all $x\in\M$, and let $B=(B_{\gamma,\gamma'})_{(\gamma,\gamma')\in\Cstar\times\Cstar}$ be
the non-negative matrix defined by:
\begin{gather*}
  \forall (\gamma,\gamma')\in\Cstar\times\Cstar\qquad B_{\gamma,\gamma'}=
  \begin{cases}
    0,&\text{if $\neg(\gamma\to \gamma')$\,,}\\
f(\gamma'),&\text{if $\gamma\to \gamma'$\,.}
  \end{cases}
\end{gather*}
Then $B$ has spectral radius\/~$1$.
\end{lemma}

\begin{proof}
  Let $\|\cdot\|$ denote the spectral radius of a non-negative matrix,
  that is to say, the greatest modulus of its eigenvalues. 

  We observe first that $B$ is a primitive non negative matrix
  (see~\cite{seneta81}). Indeed, it is irreducible since the graph of non empty
  cliques $(\Cstar,\to)$ is strongly connected on the one hand,
  according to~\cite[Lemma~3.2]{krob03}, and since $f$ is positive on
  the other hand.  And it is aperiodic since $c\to c$ holds for any
  clique $c\in\Cstar$.

  Let $h:\C\to\bbR$ be the M\"obius transform defined in~(13), and let
  $g=(g(\gamma))_{\gamma\in\Cstar}$ be the normalization vector
  defined in~(\ref{eq:30}).  The following identity is proved in
  \cite[Prop.~10.3]{abbesmair14} to hold for all $\gamma\in\Cstar$:
  $h(\gamma)=f(\gamma)g(\gamma)$\,.  It implies:
\begin{align*}
  (Bg)_\gamma&=\sum_{\gamma'\in\Cstar\tq\gamma\to\gamma'}f(\gamma')g({\gamma'})
=\sum_{\gamma'\in\Cstar\tq\gamma\to\gamma'}h(\gamma')=g(\gamma)\,.
\end{align*}

Hence $g$ is $B$-invariant on the right. Since $h>0$ on~$\Cstar$, the
vector $g$ is positive, and thus $\|B\|\geq1$\,.

We now prove the converse inequality.  For every
$\varepsilon\in(0,1)$, we consider the valuation
$f_\varepsilon:\M\to\bbR$ defined by
$f_\varepsilon(x)=\varepsilon^{|x|}f(x)$ and the matrix
$B_\varepsilon=((B_\varepsilon)_{\gamma,\gamma'})_{(\gamma,\gamma')\in\Cstar\times\Cstar}$
defined by:
\begin{gather*}
  (B_\varepsilon)_{\gamma,\gamma'}=\un_{\{\gamma\to\gamma'\}}f_\varepsilon(\gamma')\,,
\end{gather*}
and we claim that holds: $ \|B_\varepsilon\|<1$.

For proving this, let $I=(I_\gamma)_{\gamma\in\Cstar}$ and
$F_\varepsilon=((F_\varepsilon)_\gamma)_{\gamma\in\Cstar}$ be the
vectors defined by $I_\gamma=1$ and
$(F_\varepsilon)_\gamma=f_\varepsilon(\gamma)$ for all
$\gamma\in\Cstar$\,.  Denoting by $I'$ the transpose of~$I$, and
recalling that $\height(\cdot)$ denotes the height of heaps, we have:
\begin{align}
\notag
\forall k\geq0\quad  I'B_\varepsilon^k F_\varepsilon&=\sum_{\substack{x\in\M\;:\\\height(x)=k+1}}f_\varepsilon(x)\,,\\
\label{eq:41}
I'\Bigl(\sum_{k\geq0}B_\varepsilon^k\Bigr)F_\varepsilon&=\sum_{x\in\M\setminus\{0\}}f_\varepsilon(x)\,.
\end{align}

Let us show that the right member in~\eqref{eq:41} is a convergent
series. Let $w\in\M$ be a heap with the following property:
\begin{gather}
\label{eq:42}
  \forall x,y\in\M\quad|x|=|y|\wedge x\neq y\implies \up(x\cdot w)\,\cap\up(y\cdot w)=\emptyset.
\end{gather}

Such a heap exists according to the Hat Lemma
\cite[Lemma~11.2]{abbesmair14}. Since $\pr$ is a probability measure,
we have, for every integer $n\geq0$:
\begin{align*}
  \pr\Bigl(\bigcup_{x\in\M\tq|x|=n}\up(x\cdot w)\Bigr)&\leq1\\
\sum_{x\in\M\tq|x|=n}\pr(\up (x\cdot
w))&\leq1&&\text{by~(\ref{eq:42})}\\
f(w)\Bigr(\sum_{x\in\M\tq|x|=n}f(x)\Bigr)&\leq 1&&\text{since $f$ is multiplicative}
\end{align*}

It follows that there exists a constant $M>0$ such that:
\begin{gather*}
  \forall n\geq0\quad\sum_{x\in\M\tq|x|=n}f(x)\leq M<\infty
\end{gather*}
It implies:
\begin{gather*}
  \sum_{x\in\M}f_\varepsilon(x)=\sum_{n\geq0}\varepsilon^n\Bigl(\sum_{x\in\M\tq|x|=n}f(x)\Bigr)
\leq M \Bigl(\sum_{n\geq0}\varepsilon^n\Bigr)<\infty\,,
\end{gather*}
showing that the series in the right member of~(\ref{eq:41}) is
convergent. 

Now the same argument as for $B$ applies: $B_\varepsilon$~is
primitive. Henceforth, from~(\ref{eq:41}), we deduce that
$\|B_\varepsilon\|<1$, as claimed. 

Passing to the limit when $\varepsilon\to 1^-$\,, we obtain
$\|B\|\leq1$. This concludes the proof of the lemma.
\end{proof}

\begin{lemma}
  \label{lem:5}
Let $a\in\Sigma$ be a piece, and let $\M'_a$ be the sub-monoid of
$\M$ consisting of heaps with no occurrence of~$a$. Then:
\begin{align*}
  \sum_{x\in\M'_a}\pr(\up x)&<\infty\,,&
  \sum_{x\in\M'_a}|x|\,\pr(\up x)&<\infty\,,
&  \sum_{x\in\M'_a}|x|^2\,\pr(\up x)&<\infty\,.
\end{align*}
\end{lemma}

\begin{proof}
  We still denote as above by $\|\cdot\|$ the spectral radius of a
  non-negative matrix. Let $B$ the matrix defined as in
  Lemma~\ref{lem:4}, and let $B_a$ be the matrix obtained by replacing
  in $B$ all entries $(\gamma,\gamma')$ by $0$ as long as $\gamma$ or
  $\gamma'$ contains an occurrence of~$a$. Then the non-negative
  matrices $B$ and $B_a$ satisfy $B_a\leq B$ and $B_a\neq B$. Since
  $B$ is primitive, and since $\|B\|=1$ by Lemma~\ref{lem:4}, it
  follows from Perron-Frobenius theory~\cite[Chapter~1]{seneta81} that
  $\|B_a\|<1$. The result follows.
\end{proof}

\begin{lemma}
  \label{lem:6}
  Let $a\in\Sigma$ be a piece. Let $V$ be the first hitting time
  of~$a$, and let $(V_k)_{k\geq0}$ be the associated sequence of
  iterated stopping times. Fix $x\neq0$ a heap, and let
  $J_x:\BM\to\bbN\cup\{\infty\}$ be the random variable defined by:
\begin{gather*}
  J_x(\xi)=\inf\{k\geq0\tq V_k(\xi)\geq x\}\,.
\end{gather*}
Then the random variable $U_x:\BM\to\Mbar$ defined by:
\begin{align*}
  U_x(\xi)&=
  \begin{cases}
    V_{J_x(\xi)}(\xi)\,,&\text{if $J_x(\xi)<\infty$\,,}\\
\xi\,,&\text{if $J_x(\xi)=\infty$}\,,
  \end{cases}
\end{align*}
is an\/ \ast, and there exists a constant $C\geq0$, independent
of~$x$, such that:
\begin{gather}
\label{eq:36}
\esp\bigl(\,(|U_x|-|x|)^2\;\big|\up x\bigr)\leq C\,.
\end{gather}
\end{lemma}

\begin{proof}
  The fact that $U_x$ is an \ast\ is an easy consequence of the
  $V_k$'s being \ast\ (Proposition~\ref{prop:4}). Since $V$ is
  exhaustive by Proposition~\ref{prop:5}, in particular
  $J_x(\xi)<\infty$ for \pas\ every $\xi\in\up x$. Henceforth, the
  conditional expectation in~(\ref{eq:36}) is computed as the
  following sum:
  \begin{align*}
    \esp\bigl(\,(|U_x|-|x|)^2\;\big|\up x\bigr)&=\frac1{\pr(\up x)}\sum_{y\in
        \U_x}(|y|-|x|)^2\pr\bigl( \{U_x=y\}\,\cap\up x\bigr)\,,
  \end{align*}
  where $\U_x$ denotes the set of finite values assumed by~$U_x$\,.
  Since $U_x$ is an \ast, we have for all $y\in\U_x$\,:
\begin{gather*}
\{U_x=y\}=\up y\,,\\
\frac1{\pr(\up x)}\pr\bigl(\{U_x=y\}\,\cap\up x\bigr)
=\frac{\pr\bigl(\up(y\vee
  x)\bigr)}{\pr(\up x)}=\pr\bigl(\up(y-x)\bigr)\,,
\end{gather*}
the later equality since $x\leq y$ and by the multiplicativity
property of~$\pr$\,. Therefore:
\begin{align}
\label{eq:37}
  \esp\bigl(\,(|U_x|-|x|)^2\;\big|\up
  x\bigr)&=\sum_{y\in\U_x}|y-x|^2\pr\bigl(\up(y-x)\bigr)
\end{align}

Recall that each heap can be seen itself as a partially ordered
labelled set~\cite{viennot86}, where elements are labelled
by~$\Sigma$. Assume first that $x$ contains a unique maximal piece,
say $b\in\Sigma$. Such a heap is called \emph{pyramidal}.  Then for
each $y\in\U_x$\,, the heap $z=y-x$ has the following shape, for some
integer $k\geq0$\,:
$z=\delta_1\cdot\ldots\cdot\delta_{k-1}\cdot\delta_k$\,, where the
$\delta_i$'s for $i\in\{1,\ldots,k-1\}$ result from the action of the
hitting time $V$ prior to $V_k\geq x$. In particular, the $k-1$ first
heaps $\delta_i$ do not have any occurrence of~$b$\,; whereas
$\delta_k$ writes as $\delta_k=u\cdot a$ for some heap $u$ with no
occurrence of~$a$. Denoting by $\M'_a$ and $\M'_b$ respectively the
sub-monoids of $\M$ of heaps with no occurrence of $a$ and of~$b$, we
have thus $z=v\cdot u\cdot a$, for some $v\in\M'_b$ and
$u\in\M'_a$\,. Hence, from~(\ref{eq:37}), we deduce:
\begin{align*}
  \esp\bigl(\,(|U_x|-|x|)^2\;\big|\up x\bigr)&\leq
  \sum_{\substack{u\in\M'_a\\ v\in\M'_b}} (|u|+|v|+1)^2\pr(\up
  u)\cdot\pr(\up v)\cdot\pr(\up a)\,.
\end{align*}

Since $a$ and $b$ range over a finite set, it follows from
Lemma~\ref{lem:5} that the sum above in the right member is bounded by
a constant. The result~(\ref{eq:36}) follows .

We have proved the result if $x$ is pyramidal. The general case
follows since every heap $x$ writes as an upper bound
$x=x_1\vee\ldots\vee x_n$ of at most $\alpha$ pyramidal heaps, with
$\alpha$ the maximal size of cliques.
\end{proof}

\begin{lemma}
  \label{lem:7}
  Let $W$ be an\/ \ast\ such that\/ $\esp|W|<\infty$.  Let $a\in\Sigma$
  be a piece.  Let $V$ be the first hitting time of~$a$, and let\/
  $(V_k)_{k\geq0}$ be the associated sequence of iterated stopping
  times. Let $K:\BM\to\bbN\cup\{\infty\}$ be the random integer defined by:
  \begin{gather*}
K(\xi)=\inf\{k\geq0\tq V_k(\xi)\geq W(\xi)\}\,.
  \end{gather*}
Then the mapping $U:\BM\to\Mbar$ defined by:
\begin{gather*}
  U(\xi)=
  \begin{cases}
V_{K(\xi)}(\xi)\,,&\text{if $K(\xi)<\infty$,}\\
\xi\,,&\text{if $K(\xi)=\infty$}    
  \end{cases}
\end{gather*}
is an\/ \ast, and there is a constant $C\geq0$ such that:
\begin{gather*}
  \esp\bigl((|U|-|W|)^2\bigr)\leq C\,.
\end{gather*}
\end{lemma}

\begin{proof}
  Let $\W$ denote the set of finite values assumed by~$W$. Since
  $\esp|W|<\infty$, in particular $W<\infty$ $\pr$-almost surely, and
  therefore:
  \begin{align*}
    \esp\bigl((|U|-|W|)^2\bigr)&=\sum_{w\in\W}\pr(\up w)\esp\bigl((|U|-|w|)^2\;\big|\up w\bigr)\\
&=\sum_{w\in\W}\pr(\up w)\esp\bigl((|U_w|-|w|)^2\;\big|\up w\bigr)&&\text{with the
  notation $U_x$ of Lemma~\ref{lem:6}}\\
&\leq\sum_{w\in\W}\pr(\up w)C&&\text{with the constant $C$ from Lemma~\ref{lem:6}}\\
&\leq C
  \end{align*}
  The latter inequality follows from the fact that the elementary
  cylinders $\up w$\,, for $w$ ranging over~$\W$, are pairwise
  disjoint since $W$ takes different values on different such
  cylinders.  The proof of Lemma~\ref{lem:7} is complete.
\end{proof}

Finally, we will use the following elementary analytic result.

\begin{lemma}
  \label{lem:3}
  Let $(X_k)_{k\geq1}$ be a sequence of real random variables defined
  on some common probability space $(\Omega,\FFF,\pr)$, and such
  that\/ $\esp|X_k|^2\leq C<\infty$ for some constant~$C$. Then
  $\lim_{k\to\infty}X_k/k=0$ holds $\pr$-almost surely.
\end{lemma}

\begin{proof}
  Let $Y_k=X_k/k$\,.  To prove the \pas\ limit $Y_k\to0$, we use the
  following well known sufficient criterion:
  \begin{gather*}
    \forall\epsilon>0\qquad\sum_{k\geq1}\pr(|Y_k|>\epsilon)<\infty\,.
  \end{gather*}
Applying  Markov inequality yields:
\begin{gather*}
  \sum_{k\geq1}\pr(|Y_k|>\epsilon)=\sum_{k\geq1}\pr(|X_k|^2>k^2\epsilon^2)
  \leq\frac1{\epsilon^2}\sum_{k\geq1}\frac{C}{k^2}<\infty\,,
\end{gather*}
which shows the result.
\end{proof}

We now proceed with the proof of uniqueness of the limit in
Theorem~\ref{thr:2}. The setting is the following. Let $W$ be an
exhaustive \ast\ such that $\esp|W|<\infty$, let $(W_n)_{n\geq0}$ be
the associated sequence of iterated stopping times. By the first part
of the proof (\S~\ref{sec:proof-convergence}), we know that the
ergodic means $M_{W,n}\varphi$ converge \pas\ toward a constant, say
$M(W,\varphi)$.

Pick $a\in\Sigma$ a piece, and let $V$ be the first hitting time
of~$a$\,. Let $(V_n)_{n\geq0}$ be the associated sequence of iterated
stopping times, and let $M(V,\varphi)$ be the limit of the associated
ergodic means $M_{V,n}\varphi$\,. We shall prove that
$M(W,\varphi)=M(V,\varphi)$. This will conclude the proof of
Theorem~\ref{thr:2}.

We consider for each integer $j\geq0$ the following random
integer $K_j:\BM\to\bbN\cup\{\infty\}$:
\begin{align*}
  K_j(\xi)=\inf\bigl\{k\geq0\tq V_k(\xi)\geq W_j(\xi)\bigr\}\,,
\end{align*}
and the \ast\ $V'_j:\BM\to\Mbar$ defined by:
\begin{align*}
  V'_j(\xi)=
  \begin{cases}
V_{K_j(\xi)}(\xi)\,,&\text{if $K_j<\infty$\,,}\\
\xi\,,&\text{if $K_j=\infty$\,.}    
  \end{cases}
\end{align*}

Since $W_j\leq V'_j$ by construction, we put $\Delta_j=V'_j-W_j$\,, so
that $V'_j=W_j\cdot\Delta_j$\,. Then, by Lemma~\ref{lem:7}, there is a constant
$C\geq0$ such that:
\begin{gather*}
\forall j\geq0\qquad  \esp|\Delta_j|^2\leq C\,.
\end{gather*}

Hence, applying Lemma~\ref{lem:3} with $X_j=|\Delta_j|$\,, and since
$|\scal\varphi{\Delta_j}|\leq M|\Delta_j|$ if $M=\max\{
|\varphi(x)|\tq x\in\Sigma\}$\,, we
have:
\begin{align}
  \label{eq:39}
\pas\quad\lim_{j\to\infty}\frac{|\Delta_j|}j&=0\,,&\pas\quad\lim_{j\to\infty}\frac{\scal\varphi{\Delta_j}}{j}&=0\,.
\end{align}

We also have, according to the result of~\S~\ref{sec:proof-convergence}:
\begin{align}
\label{eq:33}
\pas\quad  \lim_{j\to\infty}\frac{|W_j|}j&=\esp|W|>0\,,&
\pas\quad\lim_{j\to\infty}\frac{\scal\varphi{W_j}}{|W_j|}&=M(W,\varphi)\,.
\end{align}

The ergodic means can be compared as follows:
\begin{align*}
  M_{V',\,j}\varphi-
  M_{W,j}\varphi&=\frac{\scal\varphi{W_j}+\scal\varphi{\Delta_j}}{|W_j|+|\Delta_j|}-
  \frac{\scal\varphi{W_j}}{|W_j|}\\
  &=\frac{1}{|W_j|+|\Delta_j|}\scal\varphi{\Delta_j}
  -\frac{|\Delta_j|}{|W_j|+|\Delta_j|}\cdot\frac{\scal\varphi{W_j}}{|W_j|}
\end{align*}
Using~(\ref{eq:39})(\ref{eq:33}), both terms in the right member above
go to~$0$, and therefore: $M(V',\varphi)=M(W,\varphi)$. 

But, since $\lim_{j\to\infty}K_j=\infty$, we clearly have
$M(V',\varphi)=M(V,\varphi)$, and thus finally:
$M(W,\varphi)=M(V,\varphi)$, which was to be shown. The proof of
Theorem~\ref{thr:2} is complete.

\section{A Cut-Invariant Law of Large Numbers\\ for Sub-Additive
  Functions}
\label{sec:cut-invariant-sub}

In~\S~\ref{sec:cut-invariant-law}, we have obtained a Strong law of
large numbers relative to functions of the kind
$\scal\varphi\cdot:\M\to\bbR$, which are additive by
construction---and any additive function on $\M$ is of this form. 

Interesting asymptotic quantities however are not always of this
form. For instance, the ratio between the \emph{length} and the
\emph{height} of heaps, $|x|/\height(x)$\,, has been introduced in
\cite{krob03,saheb89} as a measure of the \emph{speedup} in the
execution of asynchronous processes.

The height function is sub-additive on~$\M$: $\height(x\cdot
y)\leq\height(x)+\height(y)$. This constitutes a motivation for
extending the Strong law of large numbers to sub-additive functions.
We shall return to the computation of the speedup
in~\S~\ref{sec:computing-speedup}, after having established a
convergence result for ergodic ratios with respect to sub-additive
functions (Theorem~\ref{thr:3}).

\subsection{Statement of the Law of Large Numbers for Sub-Additive Functions}
\label{sec:statement-law-large}

As for additive functions, we face the following issues:
\begin{inparaenum}[(1)]
\item\label{item:098} Define proper ergodic ratios with respect to a
  given \ast;
\item\label{item:0912091} Prove the almost sure convergence of these
  ratios;
  \item\label{item:00022} Study the uniqueness of the limit when the
    \ast\ varies.
\end{inparaenum}

We restrict the proof of uniqueness to \emph{first hitting times}
only.

\begin{theorem}
  \label{thr:3}
  Let a heap monoid $\M=\M(\Sigma,I)$ be equipped with a Bernoulli
  measure\/~$\pr$\,, and let $\varphi:\M\to\bbR$ be a sub-additive
  function, that is to say, $\varphi$~satisfies $\varphi(x\cdot
  y)\leq\varphi(x)+\varphi(y)$ for all $x,y\in\M$\,. We assume
  furthermore that $\varphi$ is non-negative on~$\M$\,.

  Let $a\in\Sigma$ be a piece of the monoid, and let $(V_n)_{n\geq0}$
  be the sequence of iterated stopping times associated with the first
  hitting time of~$a$.  Then the ratios ${\varphi(V_n)}/{|V_n|}$
  converge\/ \pas\ as $n\to\infty$\,, toward a constant which is
  independent of the chosen piece~$a$.
\end{theorem}

We gather into two separate subsections the proof of convergence
(\S~\ref{sec:proof-convergence-1}), and the proof that the limit is
independent of the chosen piece (\S~\ref{sec:proof-uniqueness}).

\subsubsection{Proof of Convergence}
\label{sec:proof-convergence-1}

The proof is based on Kingman sub-additive Ergodic Theorem, of which
we shall use the following formulation~\cite{steele89}: let
$(\Omega,\FFF,\pr)$ be a probability space, let $T:\Omega\to\Omega$ be
a measure preserving and ergodic transformation, and let
$(g_n)_{n\geq1}$ be a sequence of integrable real-valued functions
satisfying $g_{n+m}\leq g_n+g_m\circ T^n$ for all integers
$n,m\geq1$\,. Then $g_n/n$ converge \pas\ toward a constant
$g\geq-\infty$.

\begin{lemma}
  \label{lem:2}
  If $V:\BM\to\Mbar$ is an exhaustive \ast, then the shift operator
  $\theta_V:\BM\to\BM$ which is\/ \pas\ defined on~$\BM$, is measure
  preserving and ergodic.
\end{lemma}

\begin{proof}
  To prove that $\theta_V$ is $\pr$-invariant, it is enough to show
 $\pr(\theta_V^{-1}(\up x))=\pr(\up x)$ for all heaps
  $x\in\M$. Let $x\in\M$. The equality $\xi=\xi_V\cdot\theta_V(\xi)$
  holds \pas\ since $V<\infty$ $\pr$-almost surely. Therefore,
  denoting by $\V$ the set of finite values assumed by~$V$, one has:
\begin{align*}
\pas\quad  \theta_V^{-1}(\up x)&=\bigcup_{v\in\V}\up(v\cdot x)\,.
\end{align*}

The cylinders $\up v$\,, for $v$ ranging over~$\V$\,, are pairwise
disjoint, since $V$ assumes distinct values on each of them. Hence,
passing to the probabilities and using the Bernoulli property:
\begin{align*}
  \pr\bigl(\theta_V^{-1}(\up x)\bigr)&=\pr(\up x)\sum_{v\in\V}\pr(\up
  v)=\pr(\up x)\pr(|V|<\infty)=\pr(\up x)\,.
\end{align*}
This proves that $\theta_V$ is $\pr$-invariant.

We now show the ergodicity of~$\theta_V$\,.  Let $f:\BM\to\bbR$ be a
bounded measurable and $\theta_V$-invariant function. Since $V$ is
exhaustive, $\FFF=\bigvee_{n\geq1}\FFF_{V_n}$ by
Lemma~\ref{lem:1}. Hence, by the Martingale convergence theorem
\cite[Th.~35.6]{billingsley95}:
\begin{gather}
\label{eq:35}
  f=\lim_{n\to\infty}\esp(f|\FFF_{V_n})\quad\pas
\end{gather}

Since $V_n\in\M$ with probability~$1$, the Strong Bernoulli property
(Theorem~\ref{thr:1}) implies:
\begin{align*}
  \esp(f\circ\theta_{V_n}|\FFF_{V_n})&=\esp(f)\quad\pas
\end{align*}
But, since $f$ is assumed to be $\theta_V$-invariant, and noting that
$\theta_{V_n}=(\theta_V)^n$ by construction, the above writes as:
$\esp(f|\FFF_{V_n})=\esp(f)$\,, which yields $f=\esp(f)$
by~(\ref{eq:35}), proving the ergodicity of~$\theta_V$\,.
\end{proof}

We now prove the following result, which is slightly strongest than
the convergence part in the statement of Theorem~\ref{thr:3}:

\begin{intermediate}{$(\dag)$}
  For every exhaustive \ast\ $V:\BM\to\Mbar$, if $(V_n)_{n\geq1}$ is
  the sequence of iterated stopping times associated with~$V$, the
  sequence $\varphi(V_n)/|V_n|$ is \pas\ convergent, toward a
  constant.
\end{intermediate}

Since, by Proposition~\ref{prop:5}, first hitting times are
exhaustive, this statement implies indeed the convergence statement in
Theorem~\ref{thr:3}.

For the proof of~$(\dag)$, let $g_n=\varphi(V_n)$ for $n\geq0$. An
easy induction shows that for any integers $n,m\geq0$, one has:
\begin{align*}
  V_{n+m}&=V_n\cdot(V_m\circ\theta_{V_n})\,,&\theta_{V_n}&=(\theta_V)^n\,,
\end{align*}
and thus by sub-additivity of~$\varphi$\,: $g_{n+m}\leq g_n+g_m\circ
(\theta_V)^n$\,.

The application of Kingman sub-additive Ergodic Theorem recalled above
is permitted by the measure-preserving property and the ergodicity
of~$\theta_V$ proved in Lemma~\ref{lem:2}. It implies the \pas\
convergence of $g_n/n=\varphi(V_n)/n$ toward a constant. Since
$\lim_{n\to\infty}|V_n|/n=\esp|V|$ with probability~$1$ by
Theorem~\ref{thr:2}, we deduce the \pas\ convergence of the ratios
$\varphi(V_n)/|V_n|$ as $n\to\infty$ toward a constant, which
proves~$(\dag)$.

\subsubsection{Proof of Uniqueness}
\label{sec:proof-uniqueness}

To complete the proof of Theorem~\ref{thr:3}, it remains only to show
that the limit of the ratios $\varphi(V_n)/|V_n|$ is independent of
the \ast~$V$, that is to say, of the piece for which $V$ is the first
hitting time. For this, we first show the following result:

\begin{intermediate}{$(\ddag)$}
  Let $\varphi:\M\to\bbR$ be a sub-additive and non-negative
  function. Let\/ $W:\BM\to\Mbar$ be an \ast\ such that\/
  $\esp|W|<\infty$, let\/ $(W_n)_{n\geq0}$ be the associated sequence
  of iterated stopping times, and let\/ $MW$ be the\/ \pas\ limit of\/
  $\varphi(W_n)/|W_n|$\,. Let also $V$ be the first hitting time of
  some piece~$a$, let\/ $(V_n)_{n\geq0}$ be the associated sequence of
  iterated stopping times, and let\/ $MV$ be the\/ \pas\ limit of\/
  $\varphi(V_n)/|V_n|$\,. Then\/ $MV\leq MW$\,.
\end{intermediate}

For the proof of~$(\ddag)$, we follow the same line of proof as for
the uniqueness in the proof of Theorem~\ref{thr:2}
(\S~\ref{sec:uniqueness-limit}). Using the very same notations for
$V'_n$ and~$\Delta_n$\,, we have $V'_n=W_n\cdot\Delta_n$\,, and thus:
\begin{align*}
\frac{\varphi(V'_n)}{|V'_n|}-\frac{\varphi(W_n)}{|W_n|}  &
=\underbrace{\frac{\varphi(W_n\cdot\Delta_n)-\varphi(W_n)}{|W_n|+|\Delta_n|}}_{A_n}
- \underbrace{\frac{|\Delta_n|\varphi(W_n)}{|W_n|(|W_n|+|\Delta_n|)}}_{B_n}
\end{align*}

The sub-additivity of~$\varphi$ and the existence of the \CF\
decomposition of heaps shows that $\varphi(x)\leq C_1x$ for all
$x\in\M$, and for some real constant~$C_1$\,. Therefore, using again
the sub-additivity of~$\varphi$, we obtain:
\begin{align*}
  A_n&\leq C_1\frac{|\Delta_n|}{|W_n|+|\Delta_n|}\,,&\text{and
    thus:\quad}\limsup_{n\to\infty} A_n&\leq0\,.
\end{align*}

The ratios $\varphi(W_n)/|W_n|$ being bounded since they have a finite
limit, it is clear that the terms $B_n$ converge to~$0$. We deduce:
\begin{gather*}
  \limsup_{n\to\infty}\Bigl(\frac{\varphi(V'_n)}{|V'_n|}-\frac{\varphi(W_n)}{|W_n|}\Bigr)\leq0\,.
\end{gather*}

But the ratios $\varphi(V_n)/|V_n|$ also have a limit, and clearly
$\lim\varphi(V'_n)/|V'_n|=\lim\varphi(V_n)/|V_n|$\,. Hence we obtain:
\begin{gather*}
  \lim_{n\to\infty}\frac{\varphi(V_n)}{|V_n|}\leq
  \lim_{n\to\infty}\frac{\varphi(W_n)}{|W_n|} \,,
\end{gather*}
which proves~$(\ddag)$. 

\medskip It is now clear that if both $V$ and $W$ are first hitting
times, then $MV=MW$ since $MV\leq MW$ and $MW\leq MV$ by applying
$(\ddag)$ twice. This completes the proof of Theorem~\ref{thr:3}.

\subsection{Computing the Speedup}
\label{sec:computing-speedup}

Let us define the \emph{speedup} of the pair $(\M,\pr)$\,, where $\pr$
is a Bernoulli measure on the boundary $\BM$ of a heap monoid~$\M$, as
the \pas\ limit of the inverse of the ergodic ratios:
\begin{gather*}
\pas\quad  \rho=\lim_{n\to\infty}\frac{|V_n|}{\height(V_n)}\,,
\end{gather*}
where $V$ is the first hitting time associated with some piece of the
monoid. The greater the speedup, the more the parallelism is
exploited.

Based on generating series techniques, the authors of~\cite{krob03}
obtain an expression for a similar quantity for the particular case of
uniform measures. With Bernoulli measure, we obtain a more intuitive
formula, easier to manipulate for algorithmic approximation purposes.

\begin{proposition}
  \label{prop:6}
The speedup is given by:
\begin{gather}
\label{eq:8}
\rho=\sum_{c\in\Cstar}\pi(\gamma)|\gamma|\,,
\end{gather}
where $\pi$ is the invariant measure of the Markov chain of cliques
under the probability measure\/~$\pr$\,. 
\end{proposition}

\begin{proof}
  Let $W$ be the \ast\ defined in point~\ref{item:8} of
  Proposition~\ref{prop:2}. Then $W$ is exhaustive and satisfies
  $\esp|W|<\infty$ according to Proposition~\ref{prop:5}. Let
  $(W_n)_{n\geq0}$ be the associated sequence of iterated stopping
  times. Then, since the height $\height(\cdot)$ is sub-additive, it
  follows from $(\dag)$ in~\S~\ref{sec:proof-convergence-1} that the
  ratios $\height(W_n)/|W_n|$ converge \pas\ toward a constant,
  say~$MW$. Furthermore, according to $(\ddag)$
  in~\S~\ref{sec:proof-uniqueness}, $\rho^{-1}\leq MW$\,. Hence, to
  complete the proof of the proposition, it is enough to show the
  following two points:
  \begin{enumerate}
  \item\label{item:18}
    $MW=\bigl(\sum_{\gamma\in\Cstar}\pi(\gamma)|\gamma|\bigr)^{-1}$\,.
  \item\label{item:19} $MW\leq\rho^{-1}$\,.
  \end{enumerate}

  \emph{Proof of point~\ref{item:18}.}\quad For $\xi\in\BM$ an
  infinite heap given by $\xi=(\gamma_i)_{i\geq1}$\,, let $Y_n\in\M$
  be defined for each integer $n\geq0$ by
  $Y_n=\gamma_1\cdot\ldots\cdot\gamma_n$\,. For each integer $n\geq0$,
  there is an integer $K_n$ such that $W_n=Y_{K_n}$\,, and
  $\lim_{n\to\infty}K_n=\infty$\,. Therefore:
  \begin{gather}
\label{eq:38}
MW=\lim_{n\to\infty}\frac{\height(Y_{K_n})}{|Y_{K_n}|}\,.
  \end{gather}

But we have $\height(Y_j)=j$ for each integer $j\geq1$. Therefore,
applying the Strong law of large numbers~\cite{chung60} to the ergodic Markov
chain~$(C_n)_{n\geq1}$\,, we get:
\begin{align}
\label{eq:17}
  \frac{\height(Y_j)}{|Y_j|}&=\frac j{|Y_j|}=\frac
  j{|C_1|+\dots+|C_j|}
\to_{j\to\infty}
  \Bigl(\sum_{\gamma\in\Cstar}\pi(\gamma)|\gamma|\Bigr)^{-1}\,.
\end{align}
Point~\ref{item:18} results from~(\ref{eq:38}) and~(\ref{eq:17}).

\medskip\emph{Proof of point~\ref{item:19}.}\quad For each integer
$n\geq0$, let $\tau_n=\height(V_n)$\,. Then the heap $Y_{\tau_n}$ has
same height as~$V_n$\,, and has no lesser length. Therefore the ratios
satisfy:
\begin{gather*}
\frac{\height(Y_{\tau_n})}{|Y_{\tau_n}|}=\frac{\height(V_n)}{|Y_{\tau_n}|}\leq\frac{\height(V_n)}{|V_n|}
\,.
\end{gather*}
Passing to the limit, we obtain $MW\leq \rho^{-1}$\,, completing the
proof.
\end{proof}

For the example monoid $\T=\langle a,b,c\;|\;ab=ba\rangle$ equipped
with the uniform measure $\pr$ given by $\pr(\up x)=p^{|x|}$ with
$p=(3-\sqrt5)/2$\,, the computation goes as follows. Referring to the
computations already performed in~\S~\ref{sec:direct-comp-intr}, the
invariant measure $\pi$ is:
\begin{gather*}
  \pi=\frac1{2p+1}
  \begin{pmatrix}
\begin{array}{@{}r@{\,}c@{\,}l}    p&\\p\\-3p&+&2\\3p&-&1
\end{array}
\end{pmatrix}
  \begin{array}{c}
    a\\b\\c\\ab
  \end{array}
\end{gather*}

According to Proposition~\ref{prop:6}, the speedup is:
\begin{align*}
  \rho&=\pi_a+\pi_b+\pi_c+2\pi_{ab}=\frac{5p}{2p+1}
=\frac{5(7-\sqrt5)}{22}\approx1.0827\cdots
\end{align*}

Our method allows for robust algorithmic approximation of the speedup,
through the following steps:
\begin{inparaenum}[\itshape 1.]
\item Approximating the root of the M\"obius polynomial;
\item Determining the invariant measure of the matrix~(\ref{eq:21});
\item Computing the speedup through formula~(\ref{eq:8}).
\end{inparaenum}

\printbibliography

\end{document}